\documentclass[11pt, reqno]{amsart}

\usepackage[norsk,english]{babel}
\usepackage[utf8]{inputenc}
\usepackage{amsthm}
\usepackage{amsmath}
\usepackage{amssymb}
\usepackage{mathtools}
\usepackage{enumitem}
\usepackage{hyperref}
\usepackage{cite}

\hypersetup{
  colorlinks   = true, 
  urlcolor     = blue, 
  linkcolor    = blue, 
  citecolor   = red 
}

\usepackage{xcolor}
\usepackage{pgfplots}
\pgfplotsset{compat=1.15}
\usepgfplotslibrary{fillbetween}

\usepackage[capitalise]{cleveref}
\crefname{equation}{eq.}{eqs.}
\Crefname{equation}{Equation}{Equations}

\newcommand{\F}{\mathcal{F}}
\newcommand{\Sc}{\mathcal{S}}
\newcommand{\R}{\mathbb{R}}
\newcommand{\N}{\mathbb{N}}
\newcommand{\Z}{\mathbb{Z}}

\newcommand{\jpb}[1]{\langle #1 \rangle}
\newcommand{\norm}[1]{\| #1 \|}

\newcommand{\abs}[1]{\left| #1 \right|}
\newcommand{\intR}{\int_{\R}}

\DeclareMathOperator\supp{supp}

\DeclareMathOperator\dist{dist}

\newtheorem{theorem}{Theorem}[section]
\newtheorem{lemma}[theorem]{Lemma}

\newtheorem{proposition}[theorem]{Proposition}

\newtheorem{innerassump}{Assumption}
\newenvironment{assumption}[1]
  {\renewcommand\theinnerassump{\(\mathfrak{#1}\)}\innerassump}
  {\endinnerassump}

\crefname{innerassump}{Assumption}{Assumptions}
\Crefname{innerassump}{Assumption}{Assumptions}

\theoremstyle{remark}
\newtheorem{remark}[theorem]{Remark}

\numberwithin{equation}{section}

\title[Solitary waves for fully nonlocally nonlinear equations]{Solitary waves for  dispersive equations with Coifman--Meyer nonlinearities}
\author{Johanna Ulvedal Marstrander}
\address{Department of Mathematical Sciences, NTNU -- Norwegian University of Science and Technology, 7491 Trondheim, Norway}
\email{johanna.u.marstrander@ntnu.no}
\thanks{The author acknowledges the support of the project IMod (Grant No. 325114) from the Research Council of Norway.}

\begin{document}

\begin{abstract}
Using a modified version of Weinstein's argument for constrained minimization in nonlinear dispersive equations, we prove existence of solitary waves in fully nonlocally nonlinear equations, as long as the linear multiplier is of positive and slightly higher order than the Coifman--Meyer nonlinear multiplier. It is therefore the relative order of the linear term over the nonlinear one that determines the method and existence for these types of equations. In analogy to Korteweg--De Vries-type equations and water waves in the capillary regime, smooth solutions of all amplitudes can be found. We consider two structural types of symmetric Coifman--Meyer symbols \(n(\xi-\eta,\eta)\), and show that cyclical symmetry is necessary for the existence of a functional formulation. Estimates for the solution and wave speed are given as the solutions tend to the bifurcation point of solitary waves. 
\end{abstract}
\setcounter{page}{1} 
\maketitle

\section{Introduction}
\label{sec:introduction}
We search for solitary-wave solutions to the evolution equation
\begin{equation}\label{eq:time_dependent}
    u_t + (Mu - N(u,u))_x= 0,
\end{equation}
where \(u\) is a real-valued function of time and space, and \(M\) and \(N\) are linear and bilinear Fourier multipliers of the form
\begin{equation*}
    \widehat{Mu}(\xi) = m(\xi) \hat{u}(\xi),    
\end{equation*}
and
\begin{equation}\label{eq:N}
 \widehat{N(u,u)}= \int_{\R}n(\xi-\eta, \eta) \hat{u}(\xi-\eta)\hat{u}(\eta)\,d\eta.
\end{equation}
Precise conditions on the symbols \(m\) and \(n\) will be given below. With solitary waves we shall mean solutions \((x,t)\mapsto u(x-\nu t)\) that propagate with fixed shape at velocity \(\nu\), and vanish at infinity.
 For such solutions, \cref{eq:time_dependent} becomes a pseudo-differential nonlinear equation in a single variable.
 
\subsection{Background}
The problem \eqref{eq:time_dependent} is a generalization of a vast class of uni-directional nonlinear wave equations,
\begin{equation}
    u_t + (M u - f(u))_x = 0,
    \label{eq:model_equation}
\end{equation}
arising in the study of waves in water and other dispersive media. In \cref{eq:model_equation}, \(f\) is a local nonlinear function, typically of the form \(f(u) = \abs{u}^p\) or \(f(u) = u\abs{u}^{p-1}\). Whereas this is justified in many models, including the classical Korteweg--De Vries (KdV) and Boussinesq models \cite{bona2002,lannes2013}, more exact modeling typically produces nonlinear terms that incorporate mixing in the frequencies between \(u\) and itself \cite{germain2012,ionescu2018}, so-called pseudo-products or Coifman--Meyer Fourier operators \cite{coifman1978}. Two examples of this are fully nonlinearly dispersive equations, and the Babenko formulation for traveling two-dimensional {water waves \cite{babenko1987,constantin2016,hur2019,duchene2016}}. Nonlinear frequency interactions such as in \cref{eq:N} appear naturally in signal processing too, where they describe co-linear dependencies between frequency bands \cite{cohen2014}. Note that when \(n(\xi-\eta, \eta) = 1\) and \(p=2\), \cref{eq:time_dependent,eq:model_equation} coincide; and, more generally, that \(i \xi = i(\xi - \eta) + i\eta\) becomes the pseudo-product counterpart of the Leibniz rule. Our purpose in this paper is to establish a solitary existence theory comparable to the one for \cref{eq:model_equation}, but for equations with bilinear operators \(N\).

Equations such as \eqref{eq:model_equation} have been studied extensively for positive-order operators \(M\), and more recently under assumptions allowing \(M\) to be a negative-order nonlocal operator. Solitary waves, in general, may be constructed in a number of ways, see the introduction in \cite{arnesen2023}. A variational existence theory for steady waves was developed by Turner \cite{turner1984}, based on \cite{bona1983}. Weinstein presented an 'existence and dynamical stability'-argument in \cite{weinstein1987}, where the Hamiltonian energy is minimized under an \(L^2\)-constraint using Lions' then newly developed concentration--compactness method \cite{lions1984}; this method gives the wave speed as a Lagrange multiplier. Other variational existence proofs also use quadratically constrained minimization \cite{buffoni2004,groves2011}. Weinstein's results have later been improved or generalized in several papers, including the series with \cite{albert1999,arnesen2016} and \cite{maehlen2020} leading up to this investigation. In short: when the nonlinearity is leading-order quadratic for small solutions, then the cut-off for known existence of solitary waves is \(s = \frac{1}{3}\), and for \(s\leq \frac{1}{3}\) there can be no \(H^1\) solitary-wave solutions when the symbol \(m\) is homogeneous and the nonlinearity is a pure power \cite{linares2014}. Note that equations with a pure power nonlinearity \(f\) are scalable, allowing for large waves, while variational existence proofs for equations with non-homogeneous nonlinearities usually depend on solutions being sufficiently small, as in \cite{maehlen2020}. 

Symbols \(m\) of negative order, which appear in so-called Whitham-type equations, have also been considered. Although these are below the threshold \(s = \frac{1}{3}\), they are different enough from the above case to allow for solitary waves, both small \cite{ehrnstrom2012, stefanov2020, hildrum2020}, medium-sized \cite{arnesen2023}, and highest \cite{truong2022,ehrnstrom2023}. Such differentiation of sizes and `highest' solutions do not exist in the same way for \(s > 0\), and it is an effect of regularity breakdown in the equations at a local extremum of the nonlinearity.  Minimization techniques for negative \(s\) such as in \cite{ehrnstrom2012} have been developed in \cite{duchene2018,nilsson2019,dinvay2021}, where the concentration--compactness method is applied to bidirectional systems with terms where Fourier multipliers and nonlinearities are entangled. Contrary to the present case, however, all Fourier multipliers in the above-mentioned papers are linear. Full-dispersion Whitham-type systems with surface tension have been considered in \cite{ehrnstrom2018} and \cite{ehrnstrom2022a}. 
Fully mixed nonlocal and nonlinear terms also arise in the works \cite{ehrnstrom2022a,groves2008,buffoni2018}, where a local variational reduction is used to construct solitary waves. 
\\
\begin{figure}[htb]
    \begin{tikzpicture}
        \begin{axis}[
            no markers,
            axis on top,
            axis lines=middle,
            xmin=-0.5,xmax=3.3,ymin=-2,ymax=3.3,
            xtick distance=0.5,
            ytick distance=1,
            grid style={thin,black!20},
            xlabel=$s$,
            ylabel=$r$,
            grid=minor,
            font = \normalsize]
            \addplot+[name path= A, thick, black, dotted, samples = 10] {x-1/2};
            \addplot+[name path= C, thick, dotted,black,  samples = 10] {(9*x - 3)/4};
            \addplot+[name path = D, draw = none, samples= 10] {-2};
            \addplot[black!15, on layer=pre main] fill between[of=A and D, soft clip={domain= 0.2:3.5}];
            \addplot[black!15, on layer=pre main] fill between[of=C and D, soft clip={domain= 0:0.21}];
            \node at (2.5, 2.75) {\(r = s-\frac{1}{2}\)};
            \node at (1, 2.75) {\(r=\frac{9s-3}{4}\)};
        \end{axis}
    \end{tikzpicture}
    \caption{We show existence of solitary-wave solutions for values of \(s, r\) in the gray region, where \(r=\sum_{i=1}^3 r_i\), see \cref{assump:N}. The upper bound for small \(s\) will depend on the values of each \(r_i\), not only the sum. Here, we illustrate the bound when \(r_1 = r_2 = r_3\).}
    \label{fig:rsplot}
    \end{figure}
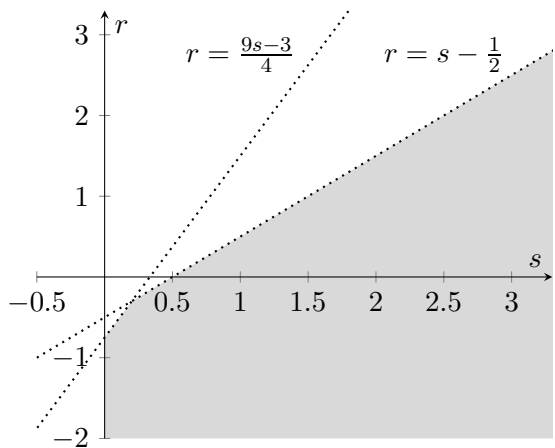

\subsection{Results}
We wish to study nonlocal nonlinearities systematically for one-dimensional model equations.  {As a first step, we start from the theory for positive-order \(M\), where our assumptions will correspond to the setting of finite depth and strong surface tension. A further discussion of the relation to existing water wave models can be found below, as well as in the Appendix.} Taking the above line of research one step further, we shall vary the nonlinearity using a bilinear Fourier operator. Postponing a precise definition of the order of \(N\), {refer to it for the moment as \(r\)}. Just as \(s= 1/3\) is a lower bound for quadratic nonlinearities in the classical case of \cref{eq:model_equation} when \(r = 0\) (so that \(N\) preserves the function space), we analogously find a relationship between the orders \(r\) and \(s\) in our case. And when the order of \(N\) is sufficiently negative, we find solitary-wave solutions also when \(0<s\leq\frac{1}{3}\). Under an assumption on \(N\) that gives rigorous meaning to the concept of \(r\), the allowed parameter ranges are visualized in \cref{fig:rsplot}. 
For such values of \(r\) and \(s\), solitary solutions are found for any size \(\mu = \frac{1}{2}\intR u^2 \,dx\). The main theorem is as follows.

\begin{theorem}[Existence of solitary-wave solutions]
For every \(\mu>0\), there is a solution \(u\in H^{\infty}(\R)\) of the solitary-wave equation, \cref{eq:solitary_wave}, satisfying \(\frac{1}{2}\norm{u}_{L^{2}}^{2} = \mu\). 
The corresponding wave speed \({\nu<m(0)}\) is subcritical, and the estimates
\begin{itemize}
    \item[(i)] \(\norm{u}_{L^{\infty}}\lesssim\norm{u}_{H^{\frac{s}{2}}}\eqsim \mu^{\frac{1}{2}}\),\\
    \item[(ii)] \(m(0) - \nu\eqsim \mu^{\beta}, \,\, \beta = \frac{s'}{2s'-1}>\frac{1}{2}\),
\end{itemize}
hold uniformly for \(0 < \mu <\mu_0\) for any fixed \(\mu_0\).
\label{thm:main}
\end{theorem}

The precise assumptions on the operators \(M\) and \(N\), including the parameter \(s'\), are given as \cref{assump:M,assump:N} below. Note that we show existence of smooth solitary waves of any size. The estimates for the size and wave speed of the solutions agree with \cite{maehlen2020} and are uniform in \(\mu\in (0,\mu_0)\) for any fixed \(\mu_0\). Thus, the implicit constants in \emph{(i), (ii)} may depend on \(\mu_0\), but \(\mu_0\) need not be chosen sufficiently small. 

\subsection{Method} To establish the existence of solitary waves, we insert the solitary-wave ansatz into \cref{eq:time_dependent} to obtain the steady equation
\begin{equation}
    -\nu u + Mu - N(u,u) = 0.
    \label{eq:solitary_wave}
\end{equation}
We can now formulate this equation as a variational problem, treating solitary-wave solutions as minimizers of the functional  
\begin{equation*}
    \mathcal{E}(u) = \underbrace{\frac{1}{2}\intR uMu\,dx}_{\eqqcolon \mathcal{M}(u)} - \underbrace{\frac{1}{3}\intR uN(u,u)\,dx}_{\eqqcolon \mathcal{N}(u)}
\end{equation*}
subject to the constraint that 
\begin{equation*}
    \mathcal{Q}(u)= \frac{1}{2}\intR u^2 \,dx = \mu.
\end{equation*}
The search for minimizers is carried out in the Sobolev space \(H^{\frac{s}{2}}(\R)\), leading to the constraint minimization problem
\begin{equation*}
    \Gamma_{\mu} = \inf \{\mathcal{E}(u)\colon u \in H^{\frac{s}{2}}(\R) \text{ and } \mathcal{Q}(u)= \mu \}.
\end{equation*}
The space \(H^{\frac{s}{2}}(\R)\) is natural; since \(M\) is of order \(s>0\), one has that \({(\mathcal{M}(u) + \mathcal{Q}(u))^{\frac{1}{2}}}\) is equivalent to the standard norm on \(H^{\frac{s}{2}}(\R)\). The sufficiently low order of \(N\) compared to \(M\) ensures that \(\mathcal{N}(u)\) is well-defined for \(u\in H^{\frac{s}{2}}(\R)\), and coercivity of  \(\mathcal{E}(u)\). According to the Lagrange multiplier principle a minimizer \(u \in H^{\frac{s}{2}}(\R)\) of \(\mathcal{E}\) then solves 
\[
-\nu \mathcal{Q}'(u) + \mathcal{E}'(u) = 0 
\]
for some \(\nu\). Since the derivatives of \(\mathcal{Q}\) and \(\mathcal{M}\) are given by \(u\) and \(Mu\), a minimizer \(u\) of \(\Gamma_{\mu}\) solves \cref{eq:solitary_wave} provided that \(\mathcal{N}'(u) = N(u,u)\). In \cref{sec:symmetry}, {we show that if there is a functional with Fréchet derivative \(N(u,u)\), then it must be \(\mathcal{N}(u)\)} and the symbol \(n\) of \(N\) {must be} cyclically symmetric in \(-\xi, \xi-\eta\) and \(\eta\). {By cyclic symmetry we mean that} 
\begin{equation}
    n(\xi-\eta, \eta) = n(\eta, -\xi) = n(-\xi, \xi-\eta).\label{eq:cyclic_symmetry}
\end{equation}
In particular, this excludes nonlinearities of the form \(uLu\), where \(L\) is a linear Fourier multiplier. For any function of three variables, \({f(\xi-\eta, \eta, -\xi)}\), a cyclically symmetric function is trivially given by the cyclic sum
\begin{equation*}
    \sum_{cyc} f(\xi-\eta, \eta, -\xi) = f(\xi-\eta, \eta, -\xi) + f(\eta, -\xi, \xi-\eta)+ f(-\xi, \xi-\eta, \eta).
\end{equation*}
Conversely, any cyclically symmetric function can be written as a cyclic sum of some function \(f(\xi-\eta, \eta, -\xi)\). Note that the variables are just cycled through; not all permutations are used. Similar symmetries arise in other works related to water waves \cite{ehrnstrom2019,ehrnstrom2022,germain2012,ionescu2018,bernicot2013,bernicot2014}.

Lions' concentration--compactness principle \cite{lions1984} is used to overcome non-compactness. We find a subsequence of any minimizing sequence that vanishes; dichotomizes; or concentrates to a solitary wave. The main challenge in our nonlocal equation is the exclusion of dichotomy: we need to show that \(M, N\) are in some sense not ``too" nonlocal, which naturally leads to regularity constraints on the symbols \(m\) and \(n\). Some additional difficulties arise from the bilinear operator \(N\), due partly to the possible negative order of \(N\), and partly to the bilinear structure.  The inclusion of \(N\) also gives rise to challenges when trying to exclude vanishing. This is in contrast to many other works, where the exclusion of vanishing is straightforward (see, e.g., \cite{maehlen2020,ehrnstrom2012,dinvay2021,nilsson2019,arnesen2016}). With the same strategy, one can show that \(\norm{u_n}_{L^{q^*}}\to 0\) if \(\{u_n\}_{n\in\N}\) vanishes if the minimizing sequence vanishes, where \(q^*\) has a sufficiently low value, but this is not enough to obtain a contradiction. 

\label{sec:assumptions_result}
To handle the theory, we assume that the symbol \(n\) is separable,
\begin{equation*}
n(\xi-\eta, \eta) = n_1(-\xi)n_2(\xi-\eta)n_3(\eta) + \text{ {cyclic permutations}},
\end{equation*} 
which means that \(N(u,u)\) is realized as a combination of linear Fourier multipliers, 
\begin{equation*}
    N(u,u) = L_1(L_2uL_3u) + \text{ {cyclic permutations.}}
\end{equation*}
The symmetric terms are added to ensure that \(n\) is cyclically symmetric as previously described. We relax this assumption below, but the cost is a stricter growth bound on \(n\). Using the Japanese bracket \(\jpb{\cdot}=(1+\cdot^2)^{1/2}\) and the standard notation that \(a\lesssim b\) means \(a\leq Cb\) for some constant \(C>0\) and \(a\simeq b\) means \(a\lesssim b\lesssim a\), these are our assumptions:
\begin{assumption}{M}\label{assump:M}
    The linear Fourier multiplier \(M\) has {non-negative}, real-valued, and even symbol \(m\in C^1(\R)\), satisfying the growth bounds
    \begin{equation*}
        \begin{split} 
    m(\xi) - m(0) &\simeq \abs{\xi}^s \,\qquad\,\text{for} \quad\abs{\xi} \geq 1,\\
    m(\xi) - m(0) &\simeq\abs{\xi}^{s'}  \qquad\text{for}\quad\abs{\xi} <1,\\
    \abs{\frac{\partial m}{\partial \xi}(\xi)} &\lesssim \jpb{\xi}^{s-1} \,\quad\text{for}\quad\xi\in\R,
        \end{split}
    \end{equation*}
    where \(s>0, \,\,s'>\frac{1}{2}\).
\end{assumption}
\begin{assumption}{N}\label{assump:N}
    The bilinear Fourier multiplier \(N\) has symbol \(n\in C^1(\R^2)\) of the form
    \begin{equation*}
        n(\xi-\eta, \eta) = \sum_{cyc} n_1({-}\xi)n_{2}(\xi-\eta)n_{3}(\eta).
    \end{equation*}
    {The symbol is even, that is invariant under the transform \((\xi, \eta)\mapsto (-\xi, -\eta)\).}
    The functions \(n_i\) are real-valued and satisfy
    \begin{align*}
        \abs{n_i(\xi)}&\lesssim \jpb{\xi}^{r_i} \quad\text{for}\quad\xi \in \R,\\
        \abs{\frac{\partial n_i(\xi)}{\partial \xi}}&\lesssim \jpb{\xi}^{r_i-1} \quad\text{for}\quad \xi \in \R,\\
        n_i(0)&\gtrsim 1,
    \end{align*}  
    where each \(r_i<\frac{s}{2}\), and
    \begin{equation*}
        \sum_{i=1}^3 r_i < s-\frac{1}{2},\qquad r_i + r_j< \frac{3s-1}{2} \quad\text{for}\quad i \neq j.
    \end{equation*}
\end{assumption}
If \(s>1\), two simple examples of \(n\) are \(n(\xi-\eta, \eta)= \jpb{\xi}^\frac{1}{6}\jpb{\xi-\eta}^{\frac{1}{6}}\jpb{\eta}^{\frac{1}{6}}\) and \(n(\xi-\eta, \eta) = \jpb{\xi}^{\frac{1}{2}} + \jpb{\xi-\eta}^{\frac{1}{2}} + \jpb{\eta}^{\frac{1}{2}}\). {See also \cref{sec:assumptions_discussion} below for a discussion of equations appearing in the literature.}
We prove \Cref{thm:main} under \cref{assump:M,assump:N} in \Cref{sec:properties,sec:functionals,sec:concentration_compactness}. In \Cref{sec:functionals,sec:concentration_compactness} we fix \(\mu>0\) and show existence of constrained minimizers  \(u\in H^{\infty}(\R)\) with \(\frac{1}{2}\norm{u}_{H^{\frac{s}{2}}}= \mu\). In \Cref{sec:properties}, we instead fix an upper bound \(\mu_0>0\) and show that \cref{thm:main} (i) and (ii) hold uniformly in \(\mu\in(0, \mu_0)\). 

\subsection{Discussion of assumptions {and related equations}}\label{sec:assumptions_discussion}
Our assumptions on the operator \(M\) are similar to those in \cite{maehlen2020}, but we allow for lower order, \(s>0\). The upper bound for the growth near zero allows us to find a sufficiently low bound for \(\Gamma_{\mu}\) in \Cref{lem:Gamma_bounds}; the lower bound near zero is only used to show properties of the solutions. The assumption on the derivative of \(m\) is to exclude dichotomy in \cref{lem:commutator}, and could be slightly weakened to \(\xi\mapsto m(\xi)/\jpb{\xi}^s\) being uniformly continuous, see \cite{maehlen2020}.

The main novelty lies in the symbol \(n\). Its inhomogeneity ensures that the contribution from the nonlinear part \(\abs{\mathcal{N}}\) is sufficiently big.  This is in turn important to show that \(\mu\mapsto\Gamma_{\mu}\) is subadditive, which is used in the concentration--compactness. 

The structure of \(n\) ensures that a variational formulation is possible and that \(N\) can be seen as a combination of linear Fourier multipliers. If \(L_1, L_2, L_3\) denote linear Fourier multipliers with symbols \(n_1(-\cdot), n_2(\cdot), n_3(\cdot)\) respectively, we have
\begin{equation*}
    \mathcal{N}(u) = \intR L_{1}uL_{2}uL_{3}u\,dx,\label{eq:funcform_B}
\end{equation*}
with linear orders \(r_1, r_2, r_3\). The upper bound on \(\sum_{i=1}^3 r_i\) gives a bound on \(\mathcal{N}\) in terms of \(\mathcal{M}\), ensuring lower boundedness of \(\Gamma_{\mu}\). Sobolev embedding as usual introduces an extra \(-1/2\), but our threshold is in line with the critical exponent below which the fractional KdV equation with \({n(\xi-\eta, \eta) = 1}, \sum_{i=1}^3 r_i = 0\) is no longer stable. The bound on the derivative of \(n\) is used both to exclude dichotomy and to ensure that the symbol \(n_i(\xi)/\jpb{\xi}^{r_i}\) defines a continuous operator on \(L^{p}(\R)\).

We also have an upper bound on the pairwise sums of \(n\). This is used to exclude vanishing as well as to show certain properties of the solution. It only comes into play when one of the \(r_i < -\frac{s}{2}\), and concerns the distribution of derivatives. The upper bounds on the sums of orders \(r_i\) are illustrated in \cref{fig:rsplot}.

{The symbols \(m, n\) are assumed to be even to avoid non-real solutions.} Finally, the assumptions that \(m\) is positive and \(n_i(0)\gtrsim 1\) are for convenience. By change of variables, it is not the signs, but the boundedness away from \(0\) that counts in both cases.

Our model is inspired by several nonlocal Boussinesq systems \cite{dinvay2022,hur2019,dinvay2021}, as well as the  {Babenko formulation \cite{babenko1987,constantin2016}} and bilinear Fourier multipliers appearing in the water wave equations \cite{bernicot2010,bernicot2013,bernicot2014,germain2012}.  It is worth noting that our theory is placed in the setting of an inhomogeneous operator on the nonlinear part (finite depth), and for operators of order in line with capillary-gravity water waves (but generally not valid for pure gravity waves). It is a goal of ours to expand the theory to cover pure gravity waves as well, as was done in \cite{ehrnstrom2012} for local nonlinearities. 

In \cite{dinvay2021}, a system is considered which, when reduced to traveling waves, is of the form here considered in all linear and quadratic parts. The extra cubic term appearing in that equation  (eq. 1.27) could probably be treated with inspiration from theory in \cite{marstrander}, but that would require a full analysis that we have not carried out in detail.
The Babenko formulation of water waves with surface tension and on finite depth involves non-polynomial nonlinearities \cite{buffoni2004}, which make the equation not directly amenable to our theory. For small solutions, they are comparable, and it is quite probable that one could relax some of our current assumptions. We give more details on the comparison of these equations with our assumptions in \cref{app:equations}.

\subsection{{An alternative assumption}}
We finish the overview of our results by presenting an alternative to \cref{assump:N}, which allows for a slightly more general form of pseudo-products.\\
\begin{assumption}{N^*}\label{assump:Nstar}
    The  bilinear Fourier multiplier \(N\) has real-valued symbol \(n\in C^{\infty}(\R^2)\)
    \begin{equation*}
        n(\xi-\eta, \eta) = \sum_{cyc} f({-}\xi, \xi-\eta, \eta),
    \end{equation*}
    Here, \(f\) is even in {\((\xi, \eta)\)} and satisfies
    \begin{align*}
       \abs{\partial^{\alpha}f(\xi, \xi-\eta, \eta)}&\lesssim \jpb{\xi}^{r_1}\jpb{\xi-\eta}^{r_2}\jpb{\eta}^{r_3} \text{ for all }\xi,\eta \in \R, \,\,\alpha\in \N_0^2,\\
        f(0, 0, 0)&\gtrsim 1,
    \end{align*}  
    where each \(r_i\leq \frac{s}{2}-\frac{1}{4}\), and
    \begin{align*}
        \sum_{i=1}^3 r_i &< s-\frac{5}{4},\\
        r_i + r_j &<\frac{3s}{2}-1 \text{ for } i\neq j.\\
    \end{align*}
\end{assumption}
The proof of \Cref{thm:main} under \cref{assump:Nstar} is carried out in \cref{sec:Bstar}. It is based on a recent result from \cite{grafakos2020}, more suitable to our case than the classical results of H\"ormander--Mikhlin and Coifman--Meyer \cite{hormander1960,coifman1978}. The growth restrictions on \(n\) under \cref{assump:Nstar} are chosen to easily reduce this case to the one covered by \cref{assump:N}. Stricter growth restrictions can always be traded for reduced regularity of \(n\), under the use of \Cref{prop:Grafakos} in \Cref{sec:Bstar}.

\section{Preliminaries}
\label{sec:functionals}
In this section, we establish properties of the functionals and minimization problem. We estimate the size of functionals and related norms for functions that are close to minimizing \(\Gamma_{\mu}\). 

\subsection{Functional-analytic setting}
Let \(\Sc(\R)\) be the Schwartz space of rapidly decaying smooth functions on \(\R\) and \(\Sc'(\R)\) its dual, the space of tempered distributions. Furthermore, let \(\F\) be the Fourier transform, defined by
\[
\F(f)(\xi) = \frac{1}{\sqrt{2\pi}}\int_{\R} f(x) \exp{(-i\xi x)}\,dx \quad \text{for }f \in \Sc(\R),
\] 
and extended by duality to \(\Sc'(\R)\). We write \(\F(f) = \hat{f}\). Define the operator \(\Lambda^{s}\colon \Sc'(\R) \to \Sc'(\R)\) by
\[
\widehat{\Lambda^{s} f}(\xi) = \jpb{\xi}^{s}\hat{f}(\xi).
\]
For \(t\in \R, 1<p<\infty\), the fractional Sobolev space, also called Bessel-potential spaces, are
\[
H^{t}_{p}(\R)= \{f\in \Sc'(\R)\colon \norm{f}_{H^{t}_{p}} = \norm{\Lambda^{t} f}_{L^{p}} <\infty\}.
\] 
We omit the subscript when \(p=2\) and write \(H^{s}(\R)\) for \(H^{s}_{2}(\R)\). While a linear Fourier multiplier is continuous on \(L^2(\R)\) as long as \({\norm{m}_{L^{\infty}}<\infty}\), the same is not true in general for \(L^p(\R)\). Different criteria guarantee continuity, one being that a linear Fourier multiplier is continuous from \(L^p(\R)\to L^p(\R), {1\leq p<\infty},\) if its symbol is continuously differentiable and {
\begin{equation*}
    \left|\partial_{\xi}^{\alpha} \sigma(\xi) \right| \lesssim \jpb{\xi}^{-\alpha}, \alpha = 0,1,
\end{equation*}
}see for example \cite[{Theorem 6.2.7}]{grafakos2014}. We define
\begin{equation*}
    H^{\infty}(\R)=\cap_{t\in \R} H^{t}(\R).
\end{equation*}

For \(t>0\), the homogeneous version of the fractional Sobolev spaces, \(\dot{H}^{t}(\R)\), are functions \(v\) with
\begin{equation*}
    \norm{v}_{\dot{H}^{t}}= \norm{\abs{\cdot}^t \hat{v}}_{L^2}<\infty.
\end{equation*}
In these spaces, we consider equivalence classes of functions, identifying those whose difference is a polynomial. Then \(\norm{v}_{\dot{H}^{t}} + \norm{v}_{L^2}\eqsim \norm{v}_{H^t}\) \cite{runst2011}.

\subsection{Bounds for the infimum \(\Gamma_{\mu}\)}
\label{sec:functionals_miniprob}
To bound \(\Gamma_{\mu}\) from below, the first step is to bound the nonlinear part \(\mathcal{N}(u)\) in terms of \(\mathcal{M}(u)\). {The reader may pay special attention to the  proof, as variations of this argument will be recurring throughout the paper. }
\begin{lemma}[Upper bound for \(\mathcal{N}(u)\)]
    Let \(u \in H^{\frac{s}{2}}(\R)\). 
    Then 
    \[
    |\mathcal{N}(u)|\lesssim \norm{u}_{L^{2}}^{3-\gamma}\norm{u}_{H^{\frac{s}{2}}}^{\gamma},
    \]
    for some \(\gamma<2.\)
    {Furthermore, for \(u,v,w \in H^{\frac{s}{2}}( \mathbb{R})\), there are \(q_{1}, q_{2}, q_{3}\geq 2\) such that \(H^{\frac{s}{2}}( \mathbb{R})\hookrightarrow H^{r_{i}}_{q_{i}}( \mathbb{R})\) and 
    \begin{equation*}
        \left|\int_{ \mathbb{R}}L_{1}uL_{2}vL_{3}w\, dx\right|\lesssim \norm{u}_{H^{r_1}_{q_1}}\norm{v}_{H^{r_2}_{q_2}}\norm{w}_{H^{r_3}_{q_3}}.
    \end{equation*}}
    \label{lem:B_upper}
\end{lemma}

\begin{proof}
    As discussed in the introduction, the separability of the symbol \(n\) allows us to write the functional \(\mathcal{N}\) as
    \begin{equation}
        \abs{\mathcal{N}(u)} = \abs{\int_{\R}L_1uL_2uL_3u \,dx}\label{step:Best_start}.
    \end{equation}
    By H\"older's inequality and using that {\(n_i(\xi)/\left\langle \xi \right\rangle ^{r_{i}}\) define continuous operators on \(L^{p}( \mathbb{R})\)},
    \begin{equation}
        \begin{split}
        \abs{\mathcal{N}(u)}\lesssim \norm{L_1u}_{L^{q_1}} \norm{L_2u}_{L^{q_2}} \norm{L_3u}_{L^{q_3}}\\
        \lesssim \norm{u}_{H^{r_1}_{q_1}}\norm{u}_{H^{r_2}_{q_2}}\norm{u}_{H^{r_3}_{q_3}},
        \end{split}
        \label{step:holder}
    \end{equation}
    where \(\sum_{i=1}^3 \frac{1}{q_i} = 1\). If \(q_i\geq 2\), Sobolev embedding yields
    \begin{equation}
        \norm{u}_{H^{r_i}_{q_i}}\lesssim \norm{u}_{H^{\frac{\tau_i}{2}}},\label{step:SobEmb}
    \end{equation}
    where \(\tau_i\) may be any number satisfying \(\tau_i\geq 2r_i\) and \(r_i + \frac{1}{2} \leq \frac{\tau_i}{2} + \frac{1}{q_i}\). If  \(\tau_i<0\), then clearly \(\norm{u}_{H^{\frac{\tau_i}{2}}}\leq \norm{u}_{L^2}\). Otherwise, we interpolate to obtain
    \begin{equation}
        \norm{u}_{H^{r_i}_{q_i}}\lesssim \norm{u}_{H^{\frac{\tau_i}{2}}} \leq \norm{u}_{L^2}^{1-\frac{\tau_i}{s}}\norm{u}_{H^{\frac{s}{2}}}^{\frac{\tau_i}{s}}.\label{step:interpolation}
    \end{equation}
    Clearly, this is only possible if \(\tau_i\leq s\). Assuming for the moment that this is the case, and combining the steps so far,
    \begin{equation*}
        \abs{\mathcal{N}(u)}\leq \norm{u}_{L^2}^{3-\frac{[\tau_1]^+ +[\tau_2]^+ +[\tau_3]^+}{s}} \norm{u}_{H^{\frac{s}{2}}}^{\frac{[\tau_1]^+ +[\tau_2]^+ +[\tau_3]^+}{s}},
    \end{equation*}
    where \([\tau_i]^+\) denotes the positive part of \(\tau_i\), i.e. \([\tau_i]^+ = \tau_i\) when \(\tau_i\geq0\) and \(0\) when \(\tau_i< 0\). If we can pick \(\tau_i\) such that
    \begin{equation*}
        \sum_{i=1}^3 [\tau_i]^+ < 2s,
    \end{equation*}
    we will have proven the lemma with \(\gamma = \frac{1}{s}(\sum_{i=1}^3 [\tau_i]^+)\).

    It is not immediately clear that one can pick \(q_i, \tau_i\) such that H\"older's inequality and the Sobolev embeddings hold, while at the same time ensuring that \(\tau_i\leq s\) and the sum \(\sum_{i=1}^3 [\tau_i]^+\) is sufficiently low. In the rest of the proof, we show that it is always possible given the assumptions on \(r_i\). Summarizing, we need to show that there are \(q_i, \tau_i\) satisfying
    \begin{gather}
        \quad q_i \geq 2, \quad r_i + \frac{1}{2} \leq \frac{\tau_i}{2} + \frac{1}{q_i} , \quad r_i\leq \frac{\tau_i}{2} \leq \frac{s}{2}  \quad\text{for } i= 1,2,3  \label{req:1}\\
        \intertext{and}
        \sum_{i=1}^3 \frac{1}{q_i} = 1,\quad \sum_{i=1}^3 [\tau_i]^+ < 2s.\label{req:2}
    \end{gather}
    
    Assume without loss of generality that \(r_1\leq r_2\leq r_3\). We split the proof into two cases. Assume first that \(r_1\leq \frac{s-1}{2}\). Let \(M>\max(2, \frac{2}{s-2r_2})\) and
    \begin{align*}
        &q_1 = M, &\tau_1 = 2r_1 + 1 - \frac{2}{M}\\
        &q_2 =\frac{2M}{M-2}, &\tau_2 = 2r_2 + \frac{2}{M},\quad\,\,\,\\
        &q_3 = 2, &\tau_3 = 2r_3.\qquad\quad\quad
   \end{align*}
   Now it is straightforward to verify that \cref{req:1} and the first part of \cref{req:2} are satisfied. The second part of \cref{req:2} is automatically satisfied if \(\tau_i\) is negative for one or more \(i\in\{1,2,3\}\) since each \(\tau_i < s\). If on the other hand \(\tau_1, \tau_2, \tau_3>0\), then 
   \begin{equation*}
       \sum_{i=1}^3 [\tau_i]^+ = \sum_{i=1}^3 \tau_i = 2r_1 + 1 - \frac{2}{M} + 2r_2 +  \frac{2}{M} + 2r_3 <2s. 
   \end{equation*}

    Now, assume that \(\frac{s-1}{2}<r_1\leq r_2 \leq r_3\). In this case, let
    \begin{align*}
        &q_1 = \frac{2}{\frac{s+1}{2} + r_1 - r_2 - r_3}, &\tau_1 = r_1 + r_2+ r_3 - \frac{s-1}{2},\\
        &q_2 = \frac{2}{\frac{s+1}{2} + r_2 - r_1 - r_3}, &\tau_2 = r_1 + r_2+ r_3 - \frac{s-1}{2},\\
        &q_3 = \frac{1}{r_3 - \frac{s-1}{2}}, &\tau_3 = s.\qquad\qquad\qquad\qquad\,\label{eq:qreqbig3}
    \end{align*}
    \Cref{req:1} is  clearly satisfied for \(i=3\). The denominators of \(q_1, q_2\) are positive and smaller than \(1\):
    \begin{equation*}
        \begin{split}
        \frac{2}{q_2}\geq \frac{2}{q_1} =\frac{s+1}{2} + r_1 - r_2 - r_3 \geq (r_1 - \frac{s-1}{2}) + (\frac{s}{2}- r_2)+ (\frac{s}{2}-r_3)>0,\\ 
        \frac{2}{q_1}\leq \frac{2}{q_2} = \frac{s+1}{2} + r_2 - r_1 - r_3 \leq (r_2 - r_3) + (\frac{s-1}{2}- r_1) + 1 <1,
        \end{split}    
    \end{equation*}
    and this ensures that \(q_i\geq 2\) for \(i=1,2\). Furthermore, \(r_i + \frac{1}{2} =\frac{\tau_i}{2} { + \frac{1}{q_{i}}}\) for \(i=1,2,\) and 
    \begin{align*}
        \tau_1 &= {\tau_2 = (r_1 -\frac{s-1}{2}) + (r_3- r_2) + 2r_2 \geq 2r_2 \geq 2r_1,}\\
        \tau_1&= \tau_2 {<} s-\frac{1}{2} - \frac{s}{2} + \frac{1}{2} = \frac{s}{2}.
    \end{align*}
    Thus \cref{req:1} is also satisfied for \(i=1,2\). We verify \cref{req:2} (as before, it suffices to check when \(\tau_i\geq 0\)):
    \begin{gather*}
        \sum_{i=1}^3 [\tau_i]^+ = \sum_{i=1}^3 \tau_i = 2(r_1 + r_2+ r_3) + 1 {<} 2s\\
        \intertext{and}
        \frac{1}{q_1} + \frac{1}{q_2} + \frac{1}{q_3} = r_3 - \frac{s-1}{2} + \frac{s+1}{2}- r_3= 1.
    \end{gather*}

    {
    To show the second estimate of the lemma, let \(u,v,w \in H^{\frac{s}{2}}( \mathbb{R})\). Again by H\"olders inequality and using that \(n_{i}(\xi)/\left\langle \xi \right\rangle^{r_{i}} \) define continuous operators on \(L^p( \mathbb{R})\), 
    \begin{equation*}
        \left|\int_{ \mathbb{R}}L_{1} u L_{2} v L_{3}w \, dx \right|\lesssim \norm{u}_{H^{r_1}_{q_1}}\norm{v}_{H^{r_2}_{q_2}}\norm{w}_{H^{r_3}_{q_3}},
    \end{equation*}
    where \(\sum_{i=1}^3 \frac{1}{q_i} = 1\). Making the same choices for \(q_{i}, \tau_{i}\) as in the argument above, we can ensure that \(q_{i}\geq 2\) and \(H^{\frac{s}{2}}( \mathbb{R})\hookrightarrow H^{\frac{\tau_{i}}{2}}( \mathbb{R})\hookrightarrow H^{r_{i}}_{q_{i}}( \mathbb{R})\).}
\end{proof}

To bound \(\Gamma_{\mu}\) from below, we need only apply the previous lemma. To find a satisfactorily low upper bound, we use a long-wave ansatz and the properties of the symbols \(m\) and \(n\) near \(0\). 

\begin{lemma}[Upper and lower bounds for \(\Gamma_{\mu}\)]
    The infimum \(\Gamma_{\mu}\) satisfies
    \begin{equation*}
        -\infty<\Gamma_{\mu}<m(0)\mu.
    \end{equation*}
    \label{lem:Gamma_bounds}
\end{lemma}
\begin{proof}
    We begin by showing the lower bound. Let \(u\in H^{\frac{s}{2}}(\R)\) with \({\mathcal{Q}(u) = \mu}\). By \cref{lem:B_upper} there is a \(\gamma<2\) and positive constants \(C_1, C_2\) such that
    \begin{align*}
    \mathcal{E}(u) &\geq \mathcal{M}(u)+ \mathcal{Q}(u) - \mathcal{Q}(u) - \abs{\mathcal{N}(u)}\\
    &> C_1\norm{u}_{H^{\frac{s}{2}}}^{2}- \mu - C_2\mu^{3 - \gamma} \norm{u}_{H^{\frac{s}{2}}}^{\gamma}\\
    &>-\infty
    \end{align*} 
    since the expression is positive as \(\norm{u}_{H^{\frac{s}{2}}}\to \infty\).

    To show the upper bound, it suffices to find  \(v\in H^{\frac{s}{2}}(\R)\) with \(\mathcal{Q}(v) = \mu\) such that 
    \begin{equation*}
        \mathcal{E}(v) < m(0)\mu.
    \end{equation*}
    To that end, pick \(\phi\in \Sc(\R)\) with \(\mathcal{Q}(\phi) = \mu, \hat{\phi}(\xi)\geq 0\) for all \(\xi\in\R\) and \(\hat{\phi}(\xi)\gtrsim 1\) for \(\abs{\xi}<1\).  Let \(0<t<1\) and define
    \begin{equation*}
    \phi_t(x) = \sqrt{t}\phi(tx).
    \end{equation*}  
    Then \(\mathcal{Q}(\phi_t) = \mu\) still. 

    We find an upper bound for \(\mathcal{M}(\phi_t)\): 
\begin{align*}
\mathcal{M}(\phi_t) = \frac{1}{2}\int_{\R} m(\xi) |\widehat{\phi_t}(\xi)|^2 \,d\xi= m(0)\mu + \frac{1}{2}\int_{\R} (m(t\xi)-m(0)) |\hat{\phi}(\xi)|^2 \,d\xi
\end{align*} 
If \(s<s'\), then \(\abs{m(t\xi) - m(0)} \lesssim |t\xi|^{s'}\) for all \(\xi \in \R\). On the other hand, if \(s\geq s'\), then \(\abs{m(t\xi) - m(0)} \lesssim |t\xi|^{s'} + |t\xi|^{s}\). Since \(t<1\) and \(\phi \in \Sc(\R)\), there must be a constant \(C_1>0\) depending on \(\phi, s'\) and \(s\) such that
\begin{equation}
\mathcal{M}(\phi_{t})
\leq m(0)\mu + C_1 t^{s'}.
\label{eq:L_bound}
\end{equation}

Next, we consider \(\mathcal{N}\). By properties of the Fourier transform,
\begin{equation*}
    \begin{split}
    \mathcal{N}(\phi_t) &= \int_{\R^2}\overline{\widehat{\phi_t}}(\xi)n_1({-}\xi)\widehat{\phi_t}(\xi-\eta)n_2(\xi-\eta)\widehat{\phi_t}(\eta)n_3(\eta)\,d\eta\,d\xi\\
    &=\int_{\R^2}t^{-\frac{3}{2}}\overline{\hat{\phi}}(\xi/t)n_1({-}\xi)\hat{\phi}((\xi-\eta)/t)n_2(\xi-\eta)\hat{\phi}(\eta/t)n_3(\eta)\,d\eta\,d\xi\\
    &= \int_{\R^2}t^{\frac{1}{2}}\overline{\hat{\phi}}(\xi)n_1({-}t\xi)\hat{\phi}(\xi-\eta)n_2(t(\xi-\eta))\hat{\phi}(\eta)n_3(t\eta)\,d\eta\,d\xi.
\end{split}
\end{equation*}
Since \(n_i(0) \gtrsim 1\) and \(n_i\) is continuous, there is a \(\delta>0\) such that \(n_i(\xi)\gtrsim 1\) for \(\abs{\xi}\leq \delta\). 
Thus \(n_1({-}\xi)n_2(\xi-\eta)n_3(\eta)\gtrsim 1\) if \(\abs{\xi}, \abs{\eta}\leq \delta/2\) (then \(\abs{\xi-\eta}\leq \delta\)). With this in mind, we split the domain of integration in the last integral into three,
\begin{equation*}
    \begin{split}
    A = \{(\xi, \eta)\in\R^2\colon \abs{\xi}\leq \frac{\delta}{2t}, \abs{\eta}\leq \frac{\delta}{2t}\},\\
    B = \{(\xi, \eta)\in\R^2\colon \abs{\xi}> \frac{\delta}{2t}, \abs{\eta}\leq \frac{\delta}{2t}\},\\
    C = \{(\xi, \eta)\in\R^2\colon \abs{\eta}> \frac{\delta}{2t}\},
    \end{split}
\end{equation*}
and consider each part separately. Since we want to find an upper bound for \(\mathcal{M}(u)-\mathcal{N}(u)\), we need a lower bound for \(\mathcal{N}\). In particular, we want to show that \(\mathcal{N}(u)\gtrsim t^{\frac{1}{2}}\). In domain \(A\), we know that \(n_1({-}t\xi)n_2(t(\xi-\eta))n_3(t\eta)\) is positive and we conclude that
\begin{equation*}
    \begin{split}
    \int_{\abs{\xi} \leq\frac{\delta}{2t}}\int_{\abs{\eta}\leq \frac{\delta}{2t}}&t^{\frac{1}{2}}\overline{\hat{\phi}}(\xi)n_1({-}t\xi)\hat{\phi}(\xi-\eta)n_2(t(\xi-\eta))\hat{\phi}(\eta)n_3(t\eta)\,d\eta\,d\xi\\
    &\gtrsim t^{\frac{1}{2}}\int_{\abs{\xi} \leq\frac{\delta}{2t}}\int_{\abs{\eta}\leq \frac{\delta}{2t}}\overline{\hat{\phi}}(\xi)\hat{\phi}(\xi-\eta)\hat{\phi}(\eta)\,d\eta\,d\xi\\
    &\gtrsim t^{\frac{1}{2}}.
\end{split}
\end{equation*}
Since the sign of \(n\) is unknown in \(B\) and \(C\), we show instead that the absolute value of the integrals over these domains approaches zero faster than \(t^{\frac{1}{2}}\) as \(t\to 0\). 
For domain \(C\), we find
\begin{equation*}
    \begin{split}
    &\left|\int_{\xi \in \R}t^{\frac{1}{2}}\int_{\abs{\eta}> \frac{\delta}{2t}}\overline{\hat{\phi}}(\xi)n_1({-}t\xi)\hat{\phi}(\xi-\eta)n_2(t(\xi-\eta))\hat{\phi}(\eta)n_3(t\eta)\,d\eta\,d\xi\right|\\
    &\qquad\quad\leq t^{\frac{1}{2}}\int_{\abs{\eta}> \frac{\delta}{2t}}\int_{\xi \in \R}\overline{\hat{\phi}}(\xi)\abs{n_1({-}t\xi)}\hat{\phi}(\xi-\eta)\abs{n_2(t(\xi-\eta))}\hat{\phi}(\eta)\abs{n_3(t\eta)}\,d\xi\,d\eta\\
    &\qquad\quad\lesssim t^{\frac{1}{2}}\int_{\abs{\eta}> \frac{\delta}{2t}}\left(\overline{\hat{\phi}}\jpb{\cdot}^{\abs{r_1}}\ast\overline{\hat{\phi}}\jpb{\cdot}^{\abs{r_2}}\right)(\eta)\hat{\phi}(\eta)\jpb{\eta}^{\abs{r_3}} \frac{\abs{\eta}}{\abs{\eta}}\,d\eta\\
    &\qquad\quad\lesssim t^{\frac{3}{2}} \int_{\abs{\eta}> \frac{\delta}{2t}}\left(\overline{\hat{\phi}}\jpb{\cdot}^{\abs{r_1}}\ast\overline{\hat{\phi}}\jpb{\cdot}^{\abs{r_2}}\right)(\eta)\hat{\phi}(\eta)\jpb{\eta}^{\abs{r_3}} \abs{\eta}\,d\eta\\
    &\qquad\quad\lesssim t^{\frac{3}{2}},
    \end{split}
\end{equation*}
where we used that \(\frac{1}{\abs{\eta}}\lesssim t\) in the domain and that \(\phi\in \Sc(\R)\). It is easy to see that a similar argument holds for the integral over domain \(B\) by observing that 
\begin{equation*}
    \begin{split}
    &\left|\int_{\abs{\xi}> \frac{\delta}{2t}}\int_{\abs{\eta}\leq \frac{\delta}{2t}}t^{\frac{1}{2}}\overline{\hat{\phi}}(\xi)n_1({-}t\xi)\hat{\phi}(\xi-\eta)n_2(t(\xi-\eta))\hat{\phi}(\eta)n_3(t\eta)\,d\eta\,d\xi\right|\\
    &\qquad\quad\leq t^{\frac{1}{2}}\int_{\abs{\xi}> \frac{\delta}{2t}}\int_{\eta\in\R}\abs{\overline{\hat{\phi}}(\xi)n_1({-}t\xi)\hat{\phi}(\xi-\eta)n_2(t(\xi-\eta))\hat{\phi}(\eta)n_3(t\eta)}\,d\eta\,d\xi.
    \end{split}
\end{equation*}
 Combining all these calculations,
\begin{equation}
    \begin{split}
    \mathcal{N}(\phi_t) &= \int_{A + B+ C}t^{\frac{1}{2}}\overline{\hat{\phi}}(\xi)n_1({-}t\xi)\hat{\phi}(\xi-\eta)n_2(t(\xi-\eta))\hat{\phi}(\eta)n_3(t\eta)\,d\eta\,d\xi\\
    &\gtrsim t^{\frac{1}{2}}-t^{\frac{3}{2}}\gtrsim t^{\frac{1}{2}}\label{eq:B_bound}.
\end{split}
\end{equation}
Combining \cref{eq:L_bound} and \cref{eq:B_bound}, then for constants \(C_1, C_2\), 
\begin{equation*}
\Gamma_{\mu}\leq \mathcal{E}(\phi_t)= \mathcal{M}(\phi_t) - \mathcal{N}(\phi_t)\leq m(0)\mu + C_1 t^{s'} - C_2 t^{\frac{1}{2}}<m(0) \mu
\end{equation*}
for \(t>0\) small enough since \(s'>\frac{1}{2}\) by assumption, which is what we wanted to show.
\end{proof}

\subsection{Estimates for near minimizers}
\label{sec:norm_estimates}
\Cref{lem:Gamma_bounds} implies that functions that are close to minimizing \(\Gamma_{\mu}\) satisfy
\begin{equation*}
    \mathcal{E}(u) < m(0)\mu.
\end{equation*}
Combining this with \cref{lem:B_upper} we can obtain certain bounds for the functionals \(\mathcal{M}, \mathcal{N}\) and associated norms. We rely on these bounds to exclude vanishing and dichotomy in the next section. 
\begin{lemma}
    Let \(\{u_n\}_{n\in\N}\) be a minimizing sequence for \(\Gamma_{\mu}\). There is a subsequence, again denoted by \(\{u_n\}_{n\in\N}\), \(\delta>0\) and some \(q^{*}\) satisfying \({q^* \in (2, \frac{2}{1-s})}\) if \(s<1\) or \(q^*>2\) otherwise, such that {
    \begin{equation*}
    \norm{u_n}_{H^{\frac{s}{2}}}^{-1},\, \mathcal{N}(u_n),\,  \norm{u_n}_{L^{q^{*}}} \geq \delta.
    \end{equation*}}\label{lem:normest}
\end{lemma}
\begin{proof}
Showing the bound for \(\norm{u_n}_{H^{\frac{s}{2}}}\) and \(\mathcal{N}(u_n)\) is straightforward, whereas showing the result for \(\norm{u_n}_{L^{q^*}}\) is more technical. We split the proof into three parts. 

\emph{Bounding \(\norm{u_n}_{H^{\frac{s}{2}}}\) from above.} In light of the discussion at the beginning of the section, \cref{lem:B_upper,lem:Gamma_bounds} together imply that 
\begin{align*}
\norm{u_n}_{H^{\frac{s}{2}}}^{2} &\eqsim \mathcal{M}(u_n) + \mathcal{Q}(u_n)\\
&{=\mathcal{E}(u_n) +\mathcal{N}(u_n) + \mathcal{Q}(u_n)}\\
&\lesssim (m(0)+1)\mu + \mu^{(3-\gamma)/2}\norm{u_n}_{H^{\frac{s}{2}}}^{\gamma},
\end{align*}
for some \(\gamma <2\) and all but finitely many \(n\in \N\). Dividing both sides by \( \norm{u_n}_{H^{\frac{s}{2}}}^{\gamma}\) it is clear that \(\norm{u_n}_{H^{\frac{s}{2}}}\) must be bounded and we conclude that
\begin{equation*}
\norm{u_n}_{H^{\frac{s}{2}}}\leq 1/\delta_1,
\end{equation*}
provided \(\delta_1>0\) is small enough.

\emph{Bounding \(\mathcal{N}(u_n)\) from below.} {\sloppy{Observe first that \(\mathcal{M}\geq m(0)\mu\) since \({m(\xi)-m(0)\geq 0}\) everywhere.} Then \cref{lem:Gamma_bounds} implies that
\begin{equation*}
    \liminf_{n\to\infty}\mathcal{N}(u_{n}) = \liminf_{n\to\infty} (\mathcal{M}(u_{n})- \mathcal{E}(u_{n}))\geq m(0)\mu - \Gamma_{\mu}>0.
\end{equation*}
}

\emph{Bounding \(\norm{u_n}_{L^{q^*}}\) from below.}
{From \cref{lem:B_upper}, we know that we can choose \(q_{1},q_{2},q_{3}\) such that
    \begin{equation}
        \abs{\mathcal{N}(u)} \lesssim \norm{u}_{H^{r_1}_{q_1}}\norm{u}_{H^{r_2}_{q_2}}\norm{u}_{H^{r_3}_{q_3}} \lesssim \norm{u}_{H^{\frac{s}{2}}}^2\norm{u}_{H^{r_i}_{q_i}} \quad \text{for} \quad i = 1,2,3.\label{eq:LqStart}
    \end{equation}
We }assume \(r_1\leq r_2\leq r_3\) and split the proof into two cases. 
Suppose first that \(r_1\leq \frac{s-1}{2}\). {In that case, we can set \(q_2 = \frac{2M}{M-2}\) for any \(M>\max(2, \frac{2}{s-2r_2})\), see the proof of \cref{lem:B_upper}.}
If \(r_2\leq 0\), then \(\norm{u}_{H^{r_2}_{q_2}}\leq \norm{u}_{L^{q_2}}\). Otherwise, let \(\theta = \frac{s-1}{2} + \frac{M-2}{2M}\), so that the Sobolev embedding \(H^{\frac{s}{2}}(\R) \hookrightarrow H^{\theta}_{q_2}(\R)\) holds. Observe that \(\frac{r_2}{\theta}<1\) since \(M> \frac{2}{s-2r_2}\). Then 
\begin{equation*}
    \norm{u}_{H^{r_2}_{q_2}}\lesssim \norm{u}_{L^{q_2}}^{1-\frac{r_2}{\theta}}\norm{u}_{H^{\theta}_{q_2}}^{\frac{r_2}{\theta}} \lesssim  \norm{u}_{L^{q_2}}^{1-\frac{r_2}{\theta}}\norm{u}_{H^{\frac{s}{2}}}^{\frac{r_2}{\theta}}.
\end{equation*}
{Combining the cases \(r_2\leq 0\) and \(r_2>0\) and inserting it into \cref{eq:LqStart} with \(i=2\),}
\begin{equation*}
    \abs{\mathcal{N}(u)} \lesssim \norm{u}_{H^{r_1}_{q_1}}\norm{u}_{H^{r_2}_{q_2}}\norm{u}_{H^{r_3}_{q_3}}\lesssim \norm{u}_{H^{\frac{s}{2}}}^{2+[\frac{r_2}{\theta}]^+}\norm{u}_{L^{q_2}}^{1-[\frac{r_2}{\theta}]^+},
\end{equation*}
so that 
\begin{equation*}
    \norm{u}_{L^{q_2}} \gtrsim \left(\delta_1^{2 + [\frac{r_2}{\theta}]^+}\delta_2\right)^\frac{1}{1-[\frac{r_2}{\theta}]^+} =\delta_3>0.
\end{equation*}
By picking \(M\) large, the estimate holds for arbitrarily small \(q^* = q_2>2\).

Now, suppose \(\frac{s-1}{2}<r_1\leq r_2\leq r_3\). Let \(a_1, a_2\) be two constants satisfying \(0\leq a_i <\frac{s}{2}- r_i\). Let
\begin{equation}
    q_1 = \frac{1}{r_1 - \frac{s-1}{2} + a_1},\quad
    q_2 = \frac{1}{r_2 - \frac{s-1}{2} + a_2},\quad
    q_3 = \frac{1}{s-r_1-r_2- a_1 - a_2}. \label{eq:req_q3}
\end{equation}
{It is easy to verify that \(\sum_{i=1}^3 \frac{1}{q_i}=1\), and that \(H^{\frac{s}{2}}(\R) \hookrightarrow L^{q_i}(\R)\) for \(i=1,2\). This gives \cref{eq:LqStart} with \(i=3\)}.  
We will show that it is possible to pick \(a_1, a_2\) so that \(q_3\) satisfies the requirements on \(q^*\) and 
\begin{equation}
    \norm{u}_{H^{r_3}_{q_3}} \lesssim \norm{u}_{L^{q_3}}^{\theta}\norm{u}_{H^{\frac{s}{2}}}^{1-\theta},\label{q_3_req}
\end{equation}
for some \(\theta>0\). Then, it follows by the same argument as in the case \(r_1\leq\frac{s-1}{2}\) that 
\begin{equation*}
    \norm{u}_{L^{q_3}}\geq \delta_4, 
\end{equation*}
for some \(\delta_4>0\). 

To have \cref{q_3_req} for a \(q_3\), we need \(H^{\tilde{s}}(\R)\hookrightarrow H^{r_3}_{q_3}(\R)\) for some \(\tilde{s}<\frac{s}{2}\). If \(q_3\) is to satify the requirements on \(q^*\), we need also that \(q_3> 2\) and \(q_3<\frac{2}{1-s}\) if \(s< 1\). Setting up the inequalities that \(q_3, r_3, s\) must then satify and inserting the value of \(q_3\) from \cref{eq:req_q3} we are left with the following criteria
\begin{equation}
    \begin{split}
    r_3 + \frac{1}{2} < \frac{s}{2} + \frac{1}{q_3} &\implies a_1+ a_2 < \frac{3s-1}{2} - (r_1+ r_2 + r_3)\\
    \frac{1}{q_3} <\frac{1}{2} &\implies a_1 + a_2 > s - r_1 - r_2 - \frac{1}{2}\\
    q_3 <\frac{2}{1-s} &\implies a_1+ a_2 < \frac{3s-1}{2} - r_1 - r_2 \,(\text{if } s<1). \end{split}
    \label{areq1}
\end{equation}
Recall that we already required \(0\leq a_i<\frac{s}{2}- r_i, i = 1,2\). This leads to additional restrictions on \(a_1 + a_2\): 
\begin{equation}
    \begin{split}
    a_1+ a_2 &<s - r_1 - r_2,\\
    a_1 + a_2 &>0.
    \end{split}\label{areq2}
\end{equation}
It is possible to find \(a_1+ a_2\) satisfying \cref{areq1,areq2} as long as 
\begin{equation*}
    \max(0, s-r_1 - r_2 - \frac{1}{2})< \min(s-r_1-r_2, \frac{3s-1}{2} - (r_1 + r_2 + r_3), \frac{3s-1}{2} - r_1 - r_2),
\end{equation*}
where the last expression on the right-hand side is only required if \(s<1\). It follows immediately from \(r_i<\frac{s}{2}, \sum_{i=1}^3 r_i <s-\frac{1}{2}, s>0\) that \(s-r_1 - r_2 - \frac{1}{2}\) is smaller than any of the expressions on the right-hand side. The assumptions on \(r_i, s\) also imply that \(s-r_1 - r_2>0\) and \(\frac{3s-1}{2} - (r_1 + r_2 + r_3)>0\). Finally, \(\frac{3s-1}{2} - r_1- r_2>0\) since
\begin{equation*}
    r_1 + r_2 < \frac{3s-1}{2}
\end{equation*}  
by assumption. Thus it is possible to obtain the desired estimate also when \(\frac{s-1}{2}<r_1\) with \(q^* = q_3\).
\end{proof}

In the remainder of the paper, we assume that minimizing sequences satisfy these estimates, passing to a subsequence if necessary.

\section{Concentration--compactness}
\label{sec:concentration_compactness} 
In this section, we prove the first part of \cref{thm:main} with the help of the concentration--compactness principle stated below. 

\begin{theorem}[The concentration--compactness principle, {\cite{lions1984}}]
    \sloppy{Any sequence of non-negative functions \({\{\rho_n\}_{n\in\N} \subset L^1(\R)}\) satisfying}
    \begin{equation*}
    \int_{\R} \rho_n \,dx = \mu,
    \end{equation*}
    for some \(\mu>0\) and for all \(\,n\in\N\), admits a subsequence \(\{\rho_{n_k}\}_{k\in\N}\) that satisfies either:
    \begin{enumerate}[label=\normalfont(\roman*)]
    \item (Compactness) There exists a subsequence \(\{y_k\}_{k\in\N} \subset \R\) such that for every \(\varepsilon>0\), there exists \(r<\infty\) satisfying 
    \begin{equation*}
    \int_{y_k-r}^{y_k+r}\rho_{n_k}(x)\,dx \geq \mu -\varepsilon \quad \text{for all }\,k\in\N.
    \end{equation*}
    \item (Vanishing) For all \(r<\infty\), 
    \begin{equation*}
    \lim_{k\to\infty}\sup_{y\in\R}\int_{y-r}^{y+r}\rho_{n_k}(x)\,dx = 0.
    \end{equation*}
    \item (Dichotomy) There exists \(\bar{\mu} \in (0,\mu)\) such that for every \({\varepsilon>0}\) there is a \({k_0 \in\N}\) and two sequences of positive \(L^1\)-functions \(\{\rho_k^{(1)}\}_{k\in\N},\{\rho_k^{(2)}\}_{k\in\N}\) satisfying for all \(k\geq k_0\)
    \begin{align*}
    \norm{\rho_{n_k} - (\rho_k^{(1)}+ \rho_k^{(2)})}_{L^1(\R)}&\leq \varepsilon,\\
    |\int_{\R} \rho_k^{(1)} \,dx  - \bar{\mu}| &\leq \varepsilon, \\
    |\int_{\R} \rho_k^{(2)} \,dx  - (\mu-\bar{\mu})| &\leq \varepsilon,\\
    \dist(\supp(\rho_k^{(1)}), \supp(\rho_k^{(2)})) &\to \infty.
    \end{align*} 
    \end{enumerate}
    \label{thm:cc}
\end{theorem}

We will apply \cref{thm:cc} to the sequence \(\{\frac{1}{2}u_n^2\}_{n\in\N}\), where \(\{u_n\}_{n\in\N}\) is a minimizing sequence for \(\Gamma_{\mu}\) satisfying the estimates in \cref{lem:normest}.

\subsection{Excluding vanishing}
\label{sec:vanishing}
The next lemma is similar to results in \cite{maehlen2020, hildrum2020}, but we avoid using \(\norm{u}_{H^{\frac{s}{2}}}^2 \eqsim \mu\). Combined with \cref{lem:normest}, it allows us to exclude vanishing. 

\begin{lemma}
Let \(v \in H^{\frac{s}{2}}(\R)\) and assume that \(q\) satisfies \(q>2\) if \(s\geq 1\) and \(q\in(2,\frac{2}{1-s})\) if \(s<1\).  Given \(\delta>0\), suppose that \(\norm{v}_{H^{\frac{s}{2}}}^{-1}, \norm{v}_{L^{q}} \geq \delta\). Then there exists \(\varepsilon>0\) such that
\begin{equation*}
\sup_{j\in \Z} \int_{j-2}^{j+2} |v(x)|^{2}\,dx \geq \varepsilon. 
\end{equation*}
\label{lem:vanish}
\end{lemma}
\begin{proof}
{Let \(\zeta\colon \R \to [0,1]\) denote a smooth function such that \({\supp \zeta \subset [-2,2]}\) and \({\sum_{j\in \Z} \zeta(x-j) = 1}\) for all \(x\in\R\) and write \({\zeta(x-j) = \zeta_j(x)}\). We will show that 
\begin{equation*}
    \sup_{j\in \mathbb{Z}} \left\|\zeta_{j}v \right\|_{L^2}^2 \geq \varepsilon.
\end{equation*}
} First, we wish to establish the inequality 
\begin{equation}
\sum_{j\in\Z} \norm{\zeta_j v}_{H^{\frac{s}{2}}}^{2}\lesssim \norm{v}_{H^{\frac{s}{2}}}^{2} \quad \text{for all } v\in H^{\frac{s}{2}}(\R).
\label{eq:van_est1}
\end{equation}
Consider the operator \(T\colon v \to \{\zeta_j v\}_{j\in\Z}\).
Clearly, \(T\) is bounded from \(H^{n}(\R)\) to \(l^{2}(H^{n}(\R))\) for any \(n\in \N_0\) since at each \(x\) only finitely many \(\zeta_jv\), \(\zeta_j v^{(n)}\) are non-zero. This implies \cref{eq:van_est1} for all \(n\in\N_0\). \Cref{eq:van_est1} follows by complex interpolation for general \(s>0\). Again using that only finitely many \(\zeta_j\) are non-zero at any \(x\), we establish that
\begin{equation}
\norm{v}_{L^{q}}^{q}\eqsim \sum_{j\in \Z} \norm{\zeta_j v}_{L^{q}}^{q} \quad \text{for all } v\in L^{q}(\R).
\label{eq:van_est2}
\end{equation} 

By assumption \(\frac{s}{2} > \frac{1}{2} - \frac{1}{q}\). Hence there is an \(\tilde{s}\) satisfying \(\frac{1}{2}-\frac{1}{q}<\frac{\tilde{s}}{2}<\frac{s}{2}\) such that by the Sobolev embedding theorem, 
\begin{equation}
\norm{\zeta_j v}_{L^{q}}^{q} \lesssim \norm{\zeta_j v}_{H^{\tilde{s}/2}}^{q}\leq \norm{\zeta_j v}_{L^{2}}^{\tau} \norm{\zeta_j v}_{H^{\frac{s}{2}}}^{q-\tau}\quad \text{for all } v\in H^{\frac{s}{2}}(\R),
\label{eq:van_est3}
\end{equation}
where \(\tau =  q(1-\frac{\tilde{s}}{s}) >0.\) If \(q -\tau< 2\), then using \cref{eq:van_est1}
\begin{equation}
    \sum_{j\in\Z}\norm{\zeta_j v}_{L^{2}}^{\tau}\norm{\zeta_j v}_{H^{\frac{s}{2}}}^{q-\tau}\leq \sum_{j\in\Z}\norm{\zeta_j v}_{L^{2}}^{q-2}\norm{\zeta_j v}_{H^{\frac{s}{2}}}^{2} \lesssim \sup_{j\in\Z}\norm{\zeta_j v}_{L^{2}}^{q-2}\norm{v}_{H^{\frac{s}{2}}}^{2}.\label{eq:van_est4}
\end{equation}
If on the other hand \(q-\tau\geq 2\), then 
\begin{equation*}
    \sum_{j\in\Z}\norm{\zeta_j v}_{H^{\frac{s}{2}}}^{q-\tau}\leq \sup_{j\in\Z}\norm{\zeta_j v}_{H^{\frac{s}{2}}}^{q-\tau-2} \left(\sum_{j\in\Z}\norm{\zeta_j v}_{H^{\frac{s}{2}}}^{2}\right) \lesssim \norm{v}_{H^{\frac{s}{2}}}^{q-\tau},
\end{equation*}
so that 
\begin{equation}
    \sum_{j\in\Z}\norm{\zeta_j v}_{L^{2}}^{\tau}\norm{\zeta_j v}_{H^{\frac{s}{2}}}^{q-\tau} \lesssim \sup_{j\in\Z}\norm{\zeta_j v}_{L^{2}}^{\tau}\norm{v}_{H^{\frac{s}{2}}}^{q-\tau}
    \label{eq:van_est5}
\end{equation}
In any case, combining \cref{eq:van_est4} or \cref{eq:van_est5} with \cref{eq:van_est2,eq:van_est3} gives  
\begin{equation*}
\norm{v}_{L^{q}}^{q} \lesssim \max(\sup_{j\in\Z}\norm{\zeta_j v}_{L^{2}}^{q-2} \norm{v}_{H^{\frac{s}{2}}}^{2},\sup_{j\in\Z}\norm{\zeta_j v}_{L^{2}}^{\tau}\norm{v}_{H^{\frac{s}{2}}}^{q-\tau}).
\end{equation*}
In light of \(\tau>0, q>2\) and \(\norm{v}_{H^{\frac{s}{2}}}^{-1}, \norm{v}_{L^{q}}\geq \delta\), this must mean that
\begin{equation*}
\sup_{j\in \Z} \int_{j-2}^{j+2} |v(x)|^{2}\,dx \geq \sup_{j\in \Z} \norm{\zeta_j v}_{L^{2}}^{2}\geq\varepsilon
\end{equation*}
for some \(\varepsilon >0\), which concludes the proof.
\end{proof}

\begin{lemma}
Vanishing does not occur.
\label{lem:vanish_result}
\end{lemma}
\begin{proof}
Let \(\delta>0\) and \(q = q^*\) be as in \cref{lem:normest}. Then \cref{lem:vanish} applied to \(\{u_n\}_{n\in\N}\) implies that vanishing does not occur.
\end{proof}

\subsection{Excluding dichotomy}
\label{sec:dichotomy}
A scaling argument shows that \(\mu\mapsto\Gamma_{\mu}\) is strictly sub-additive. 
\begin{lemma}[Sub-additivity]
\(\Gamma_{\mu}\) is strictly sub-additive:
\begin{equation*}
\Gamma_{\mu_1 + \mu_2} <\Gamma_{\mu_1} + \Gamma_{\mu_2}.
\end{equation*}
\label{lem:subadditivity}
\end{lemma}
\begin{proof}
Let \(t>1\). Let \(\{u_n\}_{n\in\N}\) be a minimizing sequence and define a new sequence \(\{\tilde{u}_n\}_{n\in\N}\) by \(\tilde{u}_n = \sqrt{t}u_n\), such that \(\mathcal{Q}(\tilde{u}_n) = tq\). We find that
\begin{align*}
\Gamma_{t\mu}&= {\liminf_{n\to\infty} (t\mathcal{M}(u_n) -t^{3/2}\mathcal{N}(u_n))}\\
&{\,\,=\liminf_{n\to\infty} (t\mathcal{E}(u_n) + (t -t^{\frac{3}{2}})\mathcal{N}(u_n))}\\
&\leq t\Gamma_{\mu} + (t-t^{\frac{3}{2}})\delta\\
&< t\Gamma_{\mu}.
\end{align*} 
In the second to last line we used that \(\mathcal{N}(u_n)\geq \delta>0\), see \cref{lem:normest}. Thus \(\mu\mapsto \Gamma_{\mu}\) is sub-homogeneous, and we proceed to show that this implies sub-additivity. Suppose first that \(\mu_1 = \mu_2\). Then the result is immediate using the established sub-homogeneity with \(t=2\). 
If on the other hand \(\mu_1 \neq \mu_2\), assume without loss of generality that \(\mu_1 <\mu_2\). Then for some \(t>1\), we can write \(\mu_2 = t\mu_1\) and hence
\begin{equation*}
\Gamma_{\mu_1 + \mu_2} = \Gamma_{\mu_2(1 + \frac{1}{t})} < (1+\frac{1}{t}) \Gamma_{\mu_2}= \Gamma_{\mu_2} + \frac{1}{t} \Gamma_{t\mu_1} { <\Gamma_{\mu_1} + \Gamma_{\mu_2}}.
\end{equation*}
\end{proof}

The key to excluding dichotomy is to show that if dichotomy occurs and \(\{u_n\}_{n\in\N}\) decomposes, then \(\{\mathcal{E}(u_n)\}_{n\in \N}\) eventually also decomposes: 
\begin{equation*}
    \Gamma_{\mu} = \lim_{n\to \infty} \mathcal{E}(u_n) =  \lim_{n\to\infty} (\mathcal{E}(u_n^{(1)}) + \mathcal{E}(u_n^{(2)})),
\end{equation*}
as this contradicts the sub-additivity just shown. Since \(M, N\) are nonlocal operators, this is not automatic. However, the regularity of \(m\) and \(n\) ensure that \(M, N\) are in some sense not "too" nonlocal. The precise statement is given in \cref{lem:commutator}. The proof follows closely \cite{maehlen2020} for \(M\), but becomes more technical for \(N\) due to an extra convolution and the possibility that \(r\leq 0\).
\begin{lemma}
    \label{lem:commutator}
    Let \(u, v \in H^{\frac{s}{2}}(\R)\) and let \(\rho\in \Sc(\R)\) be a non-negative Schwartz function. Define \(\rho_R(x) = \rho(x/R)\). The operator \(M\) satisfies
    \begin{align}
    \lim_{R\to\infty}|\int_{\R}v(\rho_R M u - M(\rho_R u))\,dx| \to 0,\label{eq:com_L1}\\
    \lim_{R\to\infty}|\int_{\R} v ((1-\rho_R) M u - M((1-\rho_R) u))\,dx| \to 0.\label{eq:com_L2}\\
    \intertext{Similarly, \(N\) satisfies}
    \lim_{R\to\infty}\left|\int_{\R}v(\rho_{R}N(u,u) - N(\rho_{R}u, u))\,dx\right| \to 0, \label{eq:com_B1}\\
    \lim_{R\to\infty}\left|\int_{\R}v((1-\rho_{R})N(u,u) - N((1-\rho_{R})u, u))\,dx\right| \to 0. \label{eq:com_B2}
    \end{align}
\end{lemma}

\begin{proof}
    \emph{Estimates for \(M\).} The mean value theorem applied to \(m\) implies that 
    \begin{align*}
        |m(\xi)- m(\eta)| &\lesssim |\xi-\eta|\sup_{|\theta| \leq |\xi|, |\eta|} \jpb{\theta}^{s-1}\\
        &\leq |\xi-\eta|(\jpb{\xi}^{s} + \jpb{\eta}^{s})\\
        &\lesssim |\xi-\eta|\jpb{\xi}^{\frac{s}{2}}\jpb{\eta}^{\frac{s}{2}}\jpb{\xi-\eta}^{\frac{s}{2}},
    \end{align*}
    for all \(\xi, t \in \R\). Combining this with Plancherel's and Fubini's theorems, \cref{eq:com_L1} follows from a direct calculation:
    \begin{align}
    &|\int_{\R} v (\rho_RM u - M(\rho_R u))\,dx|\nonumber\\
    &\leq \int_{\R} \abs{\bar{\hat{v}}(\xi)} \int_{\R}|\widehat{\rho_R}(t)||\hat{u}(\xi-t)||m(\xi-t)-m(\xi)|\,dt d\xi\nonumber\\
    &\lesssim \int_{\R} |\widehat{\rho_R}(t)||t|\jpb{t}^{\frac{s}{2}} \int_{\R}\abs{\bar{\hat{v}}(\xi)}|\hat{u}(\xi-t)|\jpb{\xi}^{\frac{s}{2}}\jpb{\xi-t}^{\frac{s}{2}}d\xi dt\nonumber\\
    &\leq \norm{v}_{H^{\frac{s}{2}}}\norm{u}_{H^{\frac{s}{2}}} \int_{\R} |R\hat{\rho}(Rt)||t|\jpb{t}^{\frac{s}{2}} dt\nonumber\\
    &= \norm{v}_{H^{\frac{s}{2}}}\norm{u}_{H^{\frac{s}{2}}}{\frac{1}{R}\int_{\R} |\hat{\rho}(\eta)||\eta|\jpb{\eta/R}^{\frac{s}{2}} \,d\eta}\label{eq:dic_lin}\\
    &\xrightarrow[]{R\to\infty} 0\nonumber
    \end{align}
    since \(\rho\in \Sc(\R)\). To show \cref{eq:com_L2}, simply observe that
    \begin{equation*}
    \int_{\R} v ((1-\rho_R)M u - M((1-\rho_R) u))\,dx= \int_{\R} v (-\rho_R Mu + M\rho_R u)\,dx\xrightarrow[]{R\to\infty} 0.
    \end{equation*}

\emph{Estimates for \(N\).} Using that \(n\) is separable, a similar calculation for \(N\) yields
\begin{align*}
    &|\int_{\R}v(\rho_{R}N(u,u) - N(\rho_{R}u, u))\,dx| \\
    &= |\int_{\R} \overline{\hat{v}(\xi)}(\hat{\rho}_{R}\ast \int_{\R}n(\cdot-\eta,\eta)\hat{u}(\cdot-\eta)\hat{u}(\eta)\,d\eta) (\xi)\\
    &\quad- \int_{\R}n(\xi-\eta, \eta)(\hat{\rho}_{R}\ast \hat{u})(\xi-\eta) \hat{u}(\eta) \,d\eta\,d\xi|
\end{align*}
This is smaller than or equal to
\begin{align*}
    \underbrace{|\int_{\R} \overline{\hat{v}}\left(\hat{\rho}_R \ast ({n_1(-\cdot)}(n_2\hat{u}\ast n_3\hat{u}))- {n_1(-\cdot)}(n_2\hat{u}\ast n_3(\hat{\rho}_R\ast\hat{u}))\right)d\xi|}_{I}&\\
    &\hspace{-1.8cm} + \text{cyclic permutations},
\end{align*}
where the the cyclic permutations are the two extra terms found by cycling through \(n_1, n_2, n_3\) in the expression for \(I\).
    We show that \(I\) approaches zero as \(R\) tends to infinity. The same argument can be applied to the other two terms, giving \cref{eq:com_B1}. Writing out the convolutions, \(I\) equals
\begin{equation*}
    \begin{split}I = \vert\int_{\R^3}\overline{\hat{v}(\xi)}\hat{\rho}_{R}(t)&\hat{u}(\eta)\hat{u}(\xi-t-\eta)n_2(\eta)\\
&\cdot({n_1(-(\xi-t))}n_3(\xi-t-\eta)-n_1(-\xi)n_3(\xi-\eta))\,d\eta\,d\xi\,dt\vert.
    \end{split}
\end{equation*}
We move the absolute value inside the integral and divide the domain into \(\abs{t}\geq R^{-\frac{1}{2}}\) and \(\abs{t}<R^{-\frac{1}{2}}\). For large \(t\) we use the triangle inequality to bound the difference to the right. For small \(t\), we use the identity \({ab-cd = c(b-d) + b(a-c)}\). Then we can estimate \(I\) by
\begin{equation}
I\leq \int_{\abs{t}\geq R^{-\frac{1}{2}}}\abs{\hat{\rho}_{R}(t)}(A + B)\,dt +  \int_{\abs{t}<R^{-\frac{1}{2}}}\abs{\hat{\rho}_{R}(t)}(C+ D)\,dt,
\label{eq:com_T_est}
\end{equation}
where
{\medmuskip=1mu \thinmuskip=1.3mu \thickmuskip=1mu 
\begin{align*}
    A&=\vert\int_{\R^2}\overline{\hat{v}(\xi)}\hat{u}(\eta)\hat{u}(\xi-t-\eta)n_2(\eta){n_1(-(\xi-t))}n_3(\xi-t-\eta)\,d\eta\,d\xi\vert,\\
    B&=\vert\int_{\R^2}\overline{\hat{v}(\xi)}\hat{u}(\eta)\hat{u}(\xi-t-\eta)n_2(\eta)n_1({-}\xi)n_3(\xi-\eta)\,d\eta\,d\xi\vert,\\
    C&=\int_{\R^2}\abs{\overline{\hat{v}(\xi)}}\abs{\hat{u}(\eta)}\abs{\hat{u}(\xi-t-\eta)}\abs{n_1({-}\xi)}\abs{n_2(\eta)}\abs{n_3(\xi-t-\eta)-n_3(\xi-\eta)}\,d\eta\,d\xi,\\
    D&=\int_{\R^2}\abs{\overline{\hat{v}(\xi)}}\abs{\hat{u}(\eta)}\abs{\hat{u}(\xi-t-\eta)}\abs{n_2(\eta)}\abs{n_3(\xi-t-\eta)}\abs{{n_1(-(\xi-t))}-n_1({-}\xi)}\,d\eta\,d\xi.
    \end{align*}}

We wish to show that \(A, B, C, D\) are bounded by polynomials in \(t\), and begin with \(A\) and \(B\). Recall that a general Fourier multiplier is bounded from \(L^q(\R)\to L^q(\R)\) if {its symbol \(\sigma\in C^{1}(\R)\) and}

\begin{equation*}
    {\abs{\sigma(\xi)}\lesssim 1, \,\text{ and }\, \abs{\frac{\partial \sigma}{\partial \xi}(\xi)}\lesssim \jpb{\xi}^{-1} \quad \text{ for all }\xi\in \R.}
\end{equation*}
{Since \(\jpb{\xi-t}^{a} \lesssim \jpb{\xi}^{a}\jpb{t}^{\abs{a}}\), the operator with symbol \(\frac{n_i(\xi-t)}{\jpb{\xi}^{r_i}\jpb{t}^{\abs{r_i} + 1}}\) is bounded from \(L^q(\R)\) to \(L^q(\R)\), uniformly in \(t\)}. Let \(\tilde{L}_{i}\) denote the operator with the translated symbol \(\tilde{n}_{i}(\xi) = n_{i}(\xi-t)\). Then
\begin{equation*}
    {\norm{\tilde{L}_iu}_{L^q}} = \jpb{t}^{\abs{r_i}+1}\norm{\F^{-1}\!\!\left(\frac{n_i(\xi-t)}{\jpb{\xi}^{r_i}\jpb{t}^{\abs{r_i}+1}}\jpb{\xi}^{r_i}\hat{u}(\xi)\right)\!\!}_{L^q}\lesssim \jpb{t}^{\abs{r_i}+1}\norm{\Lambda^{r_i}u}_{L^q}.
\end{equation*}
{Using \(\tilde{L}_{1}\) to rewrite \(A\) and applying H\"olders inequality and the continuity of \(L_{2}, L_{3}\) from \(H^{r_{1}}_{p}( \mathbb{R})\) to \(L^{p}(\mathbb{R})\), as well as the estimate for \(\tilde{L}_{i}\) above, gives {
\begin{align*}
    A \leq \int_{ \mathbb{R}}\left| L_{2}uL_{3}u \tilde{L}_{1}v \right|\,dx&\lesssim \left\|u \right\|_{H^{r_{2}}_{q_{2}}}\left\|u \right\|_{H^{r_{3}}_{q_{3}}}\|\tilde{L}_{1}v \|_{L^{q_{1}}} \\
    &\lesssim \jpb{t}^{\abs{r_1}+1} \left\|v \right\|_{H^{r_{1}}_{q_{1}}}\left\|u \right\|_{H^{r_{2}}_{q_{2}}}\norm{u}_{H^{r_{3}}_{q_{3}}},
\end{align*}}
for \(q_{i}, i=1,2,3\) such that \(\sum_{i=1}^3 \frac{1}{q_{i}} = 1.\)
As in \cref{lem:B_upper}, one may pick the \(q_{i}\)'s such that \(H^{\frac{s}{2}}( \mathbb{R})\hookrightarrow H^{r_{i}}_{q_{i}}(\mathbb{R})\), yielding the upper bound
\begin{equation*}
    A\lesssim \jpb{t}^{\abs{r_1}+1} \left\|u \right\|_{H^{\frac{s}{2}}}^{2}\norm{v}_{H^{\frac{s}{2}}} \lesssim \jpb{t}^{\max_{i=1,2,3} \abs{r_i} + 1} \left\|u \right\|_{H^{\frac{s}{2}}}^{2}\norm{v}_{H^{\frac{s}{2}}}
\end{equation*}
In the same manner, one can show that}
\begin{equation*}
    { B \lesssim \jpb{t}^{\abs{r_3}+1}\norm{u}_{H^{\frac{s}{2}}}^2\norm{v}_{H^{\frac{s}{2}}}\lesssim \jpb{t}^{\max_{i=1,2,3} \abs{r_i} + 1} \left\|u \right\|_{H^{\frac{s}{2}}}^{2}\norm{v}_{H^{\frac{s}{2}}}.}
\end{equation*}
Turning to \(C\), we can assume that \(\abs{t}\leq 1\) since \(R^{-\frac{1}{2}}\to0\) as \(R\to \infty\). By assumption on \(n_i\) and the mean value theorem, 
\begin{equation*}
    \begin{split}
    \abs{n_i(\xi-t) -n_i(\xi)} &\leq {\sup_{\theta \in [\min(\xi-t, \xi), \max(\xi-t, \xi)]} \abs{t} \left| \frac{\partial n_i}{\partial\xi}(\theta)\right|}\\
    &\lesssim \sup_{{\theta \in [\min(\xi-t, \xi), \max(\xi-t, \xi)]}} \abs{t} \jpb{\theta}^{r_i-1}\\
    &\lesssim \abs{t}\jpb{\xi}^{r_i-1}.
    \end{split}
\end{equation*}
Since {\(r_i-1<\frac{s}{2}-1<\frac{s-1}{2}\)} for \(i=1,2,3\),
\begin{align*}
    C &= \norm{\abs{\overline{\hat{v}}n_3}\abs{\hat{u}n_1}\ast \abs{{\hat{u}(\cdot -t)}(n_2(\cdot-t)-n_2(\cdot))}}_{L_1}\\
    &\lesssim \norm{\hat{v}\jpb{\cdot}^{r_3}}_{L^2}\norm{\hat{u}\jpb{\cdot}^{r_1}}_{L^2}\abs{t}\norm{\hat{u}\jpb{\cdot}^{r_2-1}}_{L^1}\\
    &\lesssim\abs{t}{\norm{u}_{H^{\frac{s}{2}}}^2\norm{v}_{H^{\frac{s}{2}}}}.
\end{align*} 
The same estimate holds for \(D\) by a similar argument. 
Having established these bounds, we insert them into \cref{eq:com_T_est} and find {
\begin{align}
    I&\lesssim  \norm{u}_{H^{\frac{s}{2}}}^2\norm{v}_{H^{\frac{s}{2}}}(\int_{\abs{t}\geq R^{-\frac{1}{2}}}\abs{\hat{\rho}_{R}(t)}\jpb{t}^{\max_{i=1,2,3} \abs{r_i} + 1}\,dt+ \int_{\abs{t}<R^{-\frac{1}{2}}}\abs{\hat{\rho}_{R}(t)}\abs{t}\,dt)\nonumber\\
    &=\norm{u}_{H^{\frac{s}{2}}}^2\norm{v}_{H^{\frac{s}{2}}}(\int_{\abs{t}\geq R^{\frac{1}{2}}}\abs{\hat{\rho}(t)}\jpb{t/R}^{\max_{i=1,2,3} \abs{r_i} + 1}\,dt+ \frac{1}{R}\int_{\abs{t}<R^{\frac{1}{2}}}\abs{\hat{\rho}(t)}\abs{t}\,dt)\nonumber\\
    &\leq \norm{u}_{H^{\frac{s}{2}}}^2\norm{v}_{H^{\frac{s}{2}}}(\int_{\abs{t}\geq R^{\frac{1}{2}}}\abs{\hat{\rho}(t)}\jpb{t}^{\max_{i=1,2,3} \abs{r_i} + 1}\,dt+ \frac{1}{R}\intR\abs{\hat{\rho}(t)}\abs{t}\,dt)\label{eq:dic_nonlin}\\
    &\xrightarrow[]{R\to \infty}0, \nonumber
\end{align}}
where the limit holds since \(\rho\in\Sc(\R)\).

\Cref{eq:com_B2} follows from a simple calculation using that \(N\) is a bilinear operator:
{\medmuskip=1mu \thinmuskip=1.3mu \thickmuskip=1mu 
\begin{align*}
|\int_{\R}&v((1-\rho_{R})N(u,u) - N((1-\rho_{R})u, u))\,dx|\\
&= |\int_{\R} - v\rho_{R}N(u,u) + vN(\rho_{R}u, u)\,dx|= |\int_{\R}v(\rho_{R}N(u,u) - N(\rho_{R}u, u))\,dx|\to 0.
\end{align*}}
\end{proof}
{
\begin{remark}
    One may apply the result to sequences \(\{u_{n}\}_{n\in \mathbb{N}}, \{v_{n}\}_{n\in \mathbb{N}}\) provided they are uniformly bounded in \(H^{\frac{s}{2}}(\mathbb{R})\). From \cref{eq:dic_lin,eq:dic_nonlin}, it is clear that the factor approching zero as \(R\to \infty\) is independent of \(u, v\). Thus we may simply do the same calculations for \(u_{n}, v_{n}\) and bound \(\left\|u_{n} \right\|_{H^{\frac{s}{2}}}, \left\|v_{n} \right\|_{H^{\frac{s}{2}}}\) by their uniform upper bounds before taking the limit. \label{rem:dichotomy}
\end{remark}
}
\Cref{lem:commutator} is used to show that \(\mathcal{E}(u_n)\) decomposes if dichotomy occurs.
\begin{lemma}
\sloppypar{Let \(\{u_n\}_{n\in\N}\) be a minimizing sequence for \(\Gamma_{\mu}\) and suppose that the sequence \(\{\frac{1}{2}u_n^2\}_{n\in\N}\) dichotomizes. {Pick smooth functions \(\phi, \psi \colon \mathbb{R}\to [0,1]\) such that}}
\begin{equation*}
\phi(x) = \begin{cases}
1 &\text{if \(|x|<1\)}\\
0 &\text{if \(|x|>2\)},
\end{cases}
\end{equation*}
{\(\phi(x) + \psi(x) = 1\) and such that \(\phi^{1/2}, \psi^{1/2}, \phi^{1/3}, \psi^{1/3}\) are also smooth (see \cref{rem:smooth_func})}.  Set \(\phi_n(x) = \phi(\frac{x-y_n}{R_n})\) and \(\psi_n(x) = \psi(\frac{x-y_n}{R_n})\), where \(\{y_n\}_{n\in\N} \subset \R\) and \(\{R_n\}_{n\in\N} \subset \R\). Define \(u_n^{(1)} = \phi_n u_n, u_n^{(2)} = \psi_n u_n\). For some \(\bar{\mu} \in (0,\mu)\), it is possible to find sequences \(\{y_n\}_{n\in\N}, \{R_n\}_{n\in\N}\) satisfying \(R_n \to \infty\) as \(n\to \infty\) such that \(\{u_n\}_{n\in\N}\) and the two sequences \(\{u_n^{(1)}\}_{n\in\N}\) and \(\{u_n^{(2)}\}_{n\in\N}\) satisfy
\begin{align}
\lim_{n\to\infty}\mathcal{Q}(u_n^{(1)})&=\bar{\mu},
\label{eq:ifdico1}\\
\lim_{n\to\infty}\mathcal{Q}(u_n^{(2)})&=(\mu-\bar{\mu}),\label{eq:ifdico2}\\
\lim_{n\to\infty}\frac{1}{2}\int_{|x-y_n|\in(R_n,2R_n)} u_n^2 \,dx &= 0,\label{eq:ifdico3}\\
\lim_{n\to\infty}(\mathcal{E}(u_n^{(1)}) + \mathcal{E}(u_n^{(2)})) &\leq \lim_{n\to\infty} \mathcal{E}(u_n).
\label{eq:ifdicoEgeq}
\end{align}
\label{lem:ifdico}
\end{lemma}
{
\begin{remark}
    Functions \(\phi, \psi\) satisfying the requirements can for example be constructed using sums, dilations, and translations starting from \(e^{-\left(\frac{x^2}{1-x^2}\right)^3}\).\label{rem:smooth_func}
\end{remark}}
\begin{proof} 
{If \(\frac{1}{2}u_{n}^2\) dichotomizes, it follows from the definition of dichotomy in \cref{thm:cc} that for every \(\varepsilon>0\) we can find \(\bar{\mu}\in(0,\mu)\), two sequences \(\{\rho_n^{(1)}\}_{n\in\N}\) and \(\{\rho_n^{(2)}\}_{n\in\N} \in L^1(\R)\), and \(n_{0}\in \mathbb{N}\) satisfying for all \(n\geq n_{0}\)
\begin{equation*}
    \begin{split}
\norm{\frac{1}{2} u_n^2 - (\rho_n^{(1)} + \rho_n^{(2)})}_{L^1}&\leq \varepsilon,\\
|\int_{\R} \rho_n^{(1)} \,dx - \bar{\mu}|&\leq \varepsilon,\\
|\int_{\R} \rho_n^{(2)} \,dx - (\mu-\bar{\mu})|&\leq \varepsilon.
    \end{split}
\end{equation*}
and \(\dist\left(\supp(\rho_{k}^{1}), \supp(\rho_{k}^{2})  \right) \to \infty\). One may also, see \cite{lions1984}, find sequences \(\{y_n\}_{n\in\N} \subset \mathbb{R}, \{R_n\}_{n\in\N} \subset \mathbb{R}\) where \(R_{n} \to \infty\) as \(n\to \infty\) such that, for \(n\geq n_{0}\),}
\begin{equation*}
    \begin{split}
\supp \rho_n^{(1)} &\subset (y_n - R_n, y_n + R_n),\\
\supp \rho_n^{(2)} &\subset (-\infty, y_n - 2R_n) {\textstyle \bigcup} (y_n + 2R_n, \infty).
    \end{split}
\end{equation*}
Then clearly 
\begin{equation*}
\frac{1}{2}\int_{|x-y_n|\in(R_n,2R_n)} u_n^2 \,dx \to 0 \quad \text{as }n\to \infty,
\end{equation*}
which is \cref{eq:ifdico3}. \Cref{eq:ifdico1} is established by a straightforward calculation,
\begin{align*}
|\mathcal{Q}(&u_n^{(1)})-\bar{\mu}| \leq |\int_{\R}\frac{1}{2}\phi_n^{2} u_n^{2}-\rho_n^{(1)}\,dx|+ |\int_{\R}\rho_n^{(1)}\,dx-\bar{\mu}|\\
&= |\int_{|x-y_n|\in(0,R_n)}\frac{1}{2} u_n^{2}-\rho_n^{(1)}-\rho_n^{(2)}\,dx| + |\int_{|x-y_n|\in(R_n,2R_n)} \frac{1}{2}\phi_n^{2}u_n^{2}\,dx|\\
&\quad + \vert\int_{\R}\rho_n^{(1)}\,dx-\bar{\mu}|\\
&\leq {\int_{\R}\left| \frac{1}{2} u_n^{2}-\rho_n^{(1)}-\rho_n^{(2)}\right|\,dx} + |\int_{|x-y_n|\in(R_n,2R_n)} \frac{1}{2}u_n^{2}\,dx|+ |\int_{\R}\rho_n^{(1)}\,dx-\bar{\mu}|\\
&\xrightarrow[]{n\to\infty} 0,
\end{align*}
and the result for \(u_n^{(2)}\), \cref{eq:ifdico2}, is found similarly.
To show \cref{eq:ifdicoEgeq}, observe that {
\begin{align*}
\mathcal{E}(u_n) &= \mathcal{M}(\phi_n u_n + \psi_n u_n) - \mathcal{N}(\phi_n u_n + \psi_n u_n)\\
&= \mathcal{M}(\phi_n u_n) + \mathcal{M}(\psi_n u_n) + \!\!\int_{\R}\phi_n u_n  M(\psi_n u_n)\,dx - \mathcal{N}(\phi_n u_n) - \mathcal{N}(\psi_n u_n)\\
&\quad - \int_{\R} \phi_n u_n N(\psi_n u_n, \psi_n u_n)\,dx  - \int_{\R}\psi_n u_n N(\phi_n u_n, \phi_n u_n).
\end{align*}}
In the last equality, we used the symmetry of \(M\) and \(N\). In light of this, it suffices to show that
\begin{align}
\lim_{n\to \infty}(\int_{\R}\phi_n u_n  M(\psi_n u_n)\,dx) &\geq 0,\label{eq:Egeq1}\\
\lim_{n\to \infty}(\int_{\R} \phi_n u_n N(\psi_n u_n, \psi_n u_n)\,dx) &= 0,\label{eq:Egeq2}\\
\lim_{n\to \infty} (\int_{\R}\psi_n u_n N(\phi_n u_n, \phi_n u_n)) &= 0.\label{eq:Egeq3}
\end{align}
{ To show \cref{eq:Egeq1} we will use \cref{lem:commutator} repeatedly. Here \(\{u_{n}\}_{n\in \mathbb{N}}\), which is uniformly bounded in \(H^{\frac{s}{2}}(\mathbb{R})\) by \cref{lem:normest}, takes the place of \(u\) and \(v\), see \cref{rem:dichotomy}.  First, we let \((\phi_n)^{\frac{1}{2}}\) and \((\psi_n)^{\frac{1}{2}}\) correspond to \(\rho_{R}\) and \(1-\rho_{R}\) respectively:}
\begin{align*}
    \lim_{n\to\infty}(\int_{\R}\phi_n u_n M(\psi_n u_n)\,dx)
    &= {\lim_{n\to\infty}(\int_{\R}(\phi_n\psi_n)^{\frac{1}{2}}u_n M((\phi_n \psi_n)^{\frac{1}{2}}u_n)\,dx)}\\ 
    &= \lim_{n\to\infty} \intR m(\xi)\abs{\widehat{((\phi_n\psi_n)^{\frac{1}{2}}u_n)}}^2\,d\xi\\
    &\geq 0.
    \end{align*}
To show \cref{eq:Egeq2}, we do the same, but with \((\phi_n)^{1/3}\) as \(\rho_{R}\) and \((\psi_n)^{1/3}\) as \((1-\rho_{R})\). Then we apply \cref{lem:B_upper}, where \(\gamma<2\). 
\begin{align*}
\lim_{n\to\infty}|\int_{\R} \phi_n u_n N(\psi_n u_n, &\psi_n u_n)\,dx|\\
&= \lim_{n\to\infty}|\mathcal{N}((\phi_n)^{\frac{1}{3}}(\psi_n)^{\frac{2}{3}} u_n)|\\
&\lesssim\lim_{n\to\infty}\norm{(\phi_n)^{\frac{1}{3}}(\psi_n)^{\frac{2}{3}} u_n}_{L^{2}}^{3-\gamma}\norm{(\phi_n)^{\frac{1}{3}}(\psi_n)^{\frac{2}{3}} u_n}_{H^{\frac{s}{2}}}^{\gamma}\\
&\lesssim\lim_{n\to\infty} \norm{u_n}_{H^{\frac{s}{2}}}^{\gamma}\lim_{n\to \infty}\left(\int_{|x-y_n|\in(R_n,2R_n)}u_n^{2}\,dx\right)^{(3-\gamma)/2}\\
&=0.
\end{align*}
\Cref{eq:Egeq3} is shown in the same manner, and we conclude that \cref{eq:ifdicoEgeq} holds if dichotomy occurs. 
\end{proof}
Combining \cref{lem:subadditivity,lem:commutator,lem:ifdico} gives rise to the aforementioned contradiction and allows us to exclude dichotomy.
\begin{lemma}
    Dichotomy does not occur.
    \label{lem:dichotomy}
\end{lemma}
\begin{proof}
{\Cref{lem:ifdico}} implies that
\begin{equation*}
{\Gamma_{\mu}= \lim_{n\to\infty}\mathcal{E}(u_n) \geq \lim_{n\to\infty}\mathcal{E}(u_n^{1})+\lim_{n\to\infty}\mathcal{E}(u_n^{2})\geq \Gamma_{\bar{\mu}} + \Gamma_{(\mu-\bar{\mu})}.}
\end{equation*}
This contradicts the subadditivity shown in \cref{lem:subadditivity}, and hence dichotomy cannot occur.
\end{proof}

\subsection{Existence of solutions}
\label{sec:existence} 
Having excluded vanishing and dichotomy, only the concentration--compactness alternative remains. We show that a minimizing sequence must then converge to a minimizer and that this minimizer is in fact a solution to the solitary-wave equation, \cref{eq:solitary_wave}.
\begin{lemma}[Existence of minimizer]
    Any minimizing sequence \(\{u_n\}_{n\in\N}\) of \(\Gamma_{\mu}\)converges, up to subsequences and translations, in \(H^{\frac{s}{2}}(\R)\) to a minimizer of \(\Gamma_{\mu}\).
    \label{lem:minimizer}
\end{lemma}
\begin{proof}
    \sloppypar{Let \(\{y_n\}_{n\in\N}\) be as in \cref{lem:ifdico} and define the translated sequence \(\tilde{u}_n(x) = u_n(x+y_n)\). Note that \({\norm{\tilde{u}_n}_{H^{\frac{s}{2}}}= \norm{u_n}_{H^{\frac{s}{2}}}<1/\delta}\) for some \(\delta>0\) by Lemma \ref{lem:normest}. Since \(\{\tilde{u}\}_{n\in\N}\) is bounded in \(H^{\frac{s}{2}}(\R)\), it admits a subsequence that converges  weakly in \(H^{\frac{s}{2}}(\R)\) to \(w\in H^{\frac{s}{2}}(\R)\).}
            
    Since the sequence concentrates, then for each \(\varepsilon>0\), there is an \(R>0\) such that
    \begin{equation*}
    \lim_{n\to\infty}\int_{{\left|x \right|>R}} \tilde{u}_n^{2}\,dx <\varepsilon.
    \end{equation*}
    If additionally \(\tilde{u}_n\) were uniformly continuous with respect to translation in \(L^{2}(\R)\), then Kolmogorov-Riez-Sudakov's compactness theorem would imply the existence of a subsequence that converges to \(\omega\) in \(L^{2}(\R)\). This is indeed the case: as \(y\to 0\),
    \begin{equation*}
        \begin{split}
        \norm{\tilde{u}_n(x+y) - \tilde{u}_n(x)}_{L^2}& = \norm{(e^{-iy\xi}-1)\hat{\tilde{u}}_n(\xi)}_{L^2}\\&\leq \sup_{\xi\in \R}|(e^{-iy\xi}-1) \jpb{\xi}^{-\frac{s}{2}}| \norm{u_n}_{H^{\frac{s}{2}}}\to 0.
        \end{split}
    \end{equation*}
    We now demonstrate that \(\mathcal{E}(w) = \Gamma_{\mu}.\)
    Fatou's lemma implies that
    \[
    \mathcal{M}(w) \leq \liminf_{n\to\infty} \mathcal{M}(\tilde{u}_n).
    \]
Furthermore, \(\mathcal{N}(\tilde{u}_n)\to \mathcal{N}(\omega)\), since 
\begin{align*}
|\mathcal{N}(\tilde{u}_n) - &\mathcal{N}(\omega)| = 
|\mathcal{N}(\tilde{u}_n - \omega) - \int_{\R}\tilde{u}_n N(\omega,\omega)\,dx +\int_{\R}\tilde{u}_n N(\omega, \tilde{u}_n)\,dx|\\
&\leq |\mathcal{N}(\tilde{u}_n -\omega)| + |\int_{\R}n(\xi-\eta, \eta)\bar{\hat{\tilde{u}}}_n(\xi)\hat{\omega}(\xi-\eta)(\hat{\omega}(\eta) - \hat{\tilde{u}}_n(\eta))\,dx|
\end{align*}
Using Lemma \ref{lem:B_upper}, where \(\gamma<2\), and the known \(L^{2}\)-convergence, 
\[
|\mathcal{N}(\tilde{u}_n -\omega)| \lesssim  \norm{\tilde{u}_n - \omega}_{L^{2}}^{3-\gamma}\norm{\tilde{u}_n-\omega}_{H^{\frac{s}{2}}}^{\gamma} \xrightarrow[n\to\infty]{} \, 0.
\]
To show that the second term also vanishes as \(n\to\infty\), we perform a similar calculation as in the proofs of Lemmas \ref{lem:B_upper} and \ref{lem:normest} for appropriate \(q_1, q_2, q_3\) and \(\tilde{s}\) such that \(r_i\leq\tilde{s}<\frac{s}{2}\):
\begin{align*}
|\int_{\R}n(\xi-\eta,& \eta)\bar{\hat{\tilde{u}}}_n(\xi)\widehat{\omega}(\xi-\eta)(\widehat{\omega}(\eta) - \widehat{\tilde{u}_n}(\eta))\,dx|\\
&\lesssim \norm{\tilde{u}_n}_{H^{\frac{s}{2}}}\norm{\tilde{\omega}_n}_{H^{\frac{s}{2}}}(\norm{\omega - \tilde{u}_n}_{H^{r_1}_{q_1}}+\norm{\omega - \tilde{u}_n}_{H^{r_2}_{q_2}}+\norm{\omega - \tilde{u}_n}_{H^{r_3}_{q_3}})\\
&\lesssim \norm{\tilde{u}_n}_{H^{\frac{s}{2}}}\norm{\tilde{\omega}_n}_{H^{\frac{s}{2}}}\norm{\omega - \tilde{u}_n}_{H^{\tilde{s}/2}}\\
&\lesssim \norm{\tilde{u}_n}_{H^{\frac{s}{2}}}\norm{\tilde{\omega}_n}_{H^{\frac{s}{2}}}\norm{\omega - \tilde{u}_n}_{L^2}^{1-\tilde{s}/s}\norm{\omega - \tilde{u}_n}_{H^{\frac{s}{2}}}^{\tilde{s}/s}\\
&\xrightarrow[n\to\infty]{} \,0
\end{align*}
since \(\tilde{u}_n \to \omega\) in \(L^{2}(\R)\).

We conclude that \(\mathcal{E}(\omega) = \Gamma_{\mu} = \lim_{n\to \infty} \mathcal{E}(\tilde{u}_n)\). Since \(\mathcal{N}(\tilde{u}_n) \to \mathcal{N}(\omega)\), then we must have \(\mathcal{M}(\tilde{u}_n) \to \mathcal{M}(\omega)\). Combined, this and the weak convergence of \(\tilde{u}_n\rightharpoonup \omega\) implies norm convergence in \(H^{\frac{s}{2}}(\R)\). 
\end{proof}

It is easily verifiable that the Fréchet derivatives \(D\mathcal{M}(u), D\mathcal{Q}(u)\) and \(D\mathcal{N}(u)\) are \((Mu, u)\) and \(N(u,u)\) respectively. Thus, minimizers of \(\Gamma_{\mu}\) are solutions of the solitary-wave equation \eqref{eq:solitary_wave} by the Lagrange multiplier principle. 
\begin{lemma}
Any minimizer \(u\in H^{\frac{s}{2}}(\R)\) of the constrained variational problem \(\Gamma_{\mu}\) solves the solitary-wave equation \eqref{eq:solitary_wave}, where \(\nu\) is the Lagrange multiplier. 
\label{lem:solutions}
\end{lemma}

\subsection{Properties and regularity of solutions.}
We end \cref{sec:concentration_compactness} by showing that solutions \(u\) satisfy \(u\in H^{\infty}(\R)\), and that the wave speed \(\nu\) is subcritical.  That will conclude the proof of the first part of \cref{thm:main}.
\begin{lemma}[Subcritical wave-speed]
    Any minimizer of \(\Gamma_{\mu}\) solves \cref{eq:solitary_wave} with subcritical wave speed \({\nu<m(0)}\).\label{lem:wave_speed_all}
\end{lemma}
\begin{proof}
    By \cref{lem:solutions}, a minimizer \(u\) of \(\Gamma_{\mu}\) solves the solitary wave equation
    \begin{equation*}
        -\nu u + M(u) - N(u,u) = 0,
    \end{equation*}
    where \(\nu\) is the wave speed. Multiplying by \(u\) and integrating gives
    \begin{equation*}
        - 2\nu \mu + 2\mathcal{M}(u)- 3\mathcal{N}(u) = 0,
    \end{equation*}
    which implies that
    \begin{equation*}
    \nu = \frac{\mathcal{E}(u) - \frac{1}{2}\mathcal{N}(u)}{\mu}<m(0),\label{quad_c_start} 
    \end{equation*}
    since {\(\mathcal{E}(u)<m(0)\mu\)} by \cref{lem:Gamma_bounds} and \(\mathcal{N}(u)>0\) by \cref{lem:normest}.
\end{proof}

To show the regularity of the solutions, we use a fractional product rule:
\begin{proposition}[\!\!{\cite[Theorem 4.6.1/1]{runst2011}}]
    Assume that 
    \begin{equation*}
    t_1 \leq t_2 \quad \text{and} \quad t_1 + t_2 > 0.
    \end{equation*}
    Furthermore, let \(f\in H^{t_1}(\R), g\in H^{t_2}(\R)\). 
    \begin{enumerate}[label = \emph{(\roman*)}]
    \item Let \(t_2> 1/2\). Then
    \[
    \norm{fg}_{H^{t_1}}\lesssim \norm{f}_{H^{t_1}}\norm{g}_{H^{t_2}}.
    \]
    \item Let \(t_2 <1/2\). Then
    \[
    \norm{fg}_{H^{t_1 + t_2 - 1/2}} \lesssim \norm{f}_{H^{t_1}}\norm{g}_{H^{t_2}}.
    \]
    \end{enumerate}
    \label{prop:product_rule}
    \end{proposition}

\begin{lemma}[Regularity of solutions]
    Any solution \(u\in H^{\frac{s}{2}}(\R)\) of \cref{eq:solitary_wave} with \(\mathcal{Q}(u) = \mu\) is also in \(H^{\infty}(\R)\). 
    \label{lem:regularity_all}
\end{lemma}

\begin{proof}
    Rewriting \cref{eq:solitary_wave}, we have
    \begin{equation}
    (M-\nu)u = N(u,u).
    \label{eq:reg_start}
    \end{equation}
    Since \(\nu<m(0)\), then \((M-\nu)\colon H^{t}(\R)\to H^{t - s}(\R) \) is an invertible linear operator with continuous inverse \((M-\nu)^{-1}\colon H^{t}(\R)\to H^{t + s}(\R)\). We wish to show that \(N(u,u)\in H^{\alpha_0}(\R)\) for some \(\alpha_0\) which is stricly greater than \(-\frac{s}{2}\). Observe first that 
    \begin{align*}
        \norm{N(u,u)}_{H^{\alpha_0}} &\simeq \norm{\jpb{\xi}^{\alpha_0}\intR \hat{u}(\xi-\eta)\hat{u}(\eta)\sum_{cyc}(n_1({-}\xi)n_2(\xi-\eta)n_3(\eta))\,d\eta}_{L^2}\\
        &\leq \norm{\Lambda^{\alpha_0}L_1(L_2uL_3u)}_{L^2}+\norm{\Lambda^{\alpha_0}L_2(L_3uL_1u)}_{L^2}\\
        &\quad+ \norm{\Lambda^{\alpha_0}L_3(L_1uL_2u)}_{L^2}.
    \end{align*}
    Therefore, we proceed to find an upper bound for terms on the form
    \begin{equation}
        \norm{\Lambda^{\alpha_0}L_i(L_juL_ku)}_{L^2} \simeq \norm{L_juL_ku}_{H^{\alpha_0 + r_i}}.
        \label{eq:reg}
    \end{equation}
    {We consider separately the cases when the smallest of \(r_{k}, r_{j}\) is larger or equal to \(\frac{s-1}{2}\), and smaller than \(\frac{s-1}{2}\).} We will apply \cref{prop:product_rule} to \cref{eq:reg} with \(f=L_ju, g= L_k u\). We assume without loss of generality that \(r_k\leq r_j\). {If \(r_k\geq\frac{s-1}{2}\), we pick some \(0<\varepsilon_{0} <\frac{s}{4}\) and apply \cref{prop:product_rule} (ii) with
    \begin{equation*}
        t_1 = \frac{s}{2} - r_j-\varepsilon_{0},\,\,t_2 = \frac{s}{2} - r_k-\varepsilon_{0}, \text{ and } \alpha_{0,1} = s- \frac{1}{2}-(r_{i} + r_{j} + r_{k})-2\varepsilon_{0}.
    \end{equation*}}
    and if \(r_k< \frac{s-1}{2}\), we apply \cref{prop:product_rule} (i) with 
    \begin{equation*}
        t_1 = r_i + \alpha_{0,2},\,\,t_2= \frac{s}{2} -r_k, \text{ and } \frac{-s}{2}+ r_k - r_i <\alpha_{0,2} \leq \frac{s}{2} - r_j - r_i.
    \end{equation*}
     Observing that {\(t_1 \leq\frac{s}{2} - r_j, t_2 =\frac{s}{2} - r_k\)}, we conclude that
    \begin{equation*}
        \norm{N(u,u)}_{H^{\alpha_0}} \lesssim \norm{u}_{H^{\frac{s}{2}}}^2,
    \end{equation*}
    where {\(\alpha_0 =\min(\alpha_{0,1}, \alpha_{0,2})\).} Combined with \cref{eq:reg_start}, this implies that \((M-\nu)u\) is in \(H^{\alpha_0}(\R)\) which in turn implies \(u\in H^{\alpha_0 + s}(\R)\). {Clearly \({\alpha_{0,1}>-2\varepsilon_{0} >-\frac{s}{2}}\), and \(\alpha_{0,2}\) can be chosen to be stricly greater than \(-\frac{s}{2}\) since \({r_i + r_j <s}\)}. Thus \(\alpha_0 + s\) is strictly bigger than \(\frac{s}{2}\). 
    Now we can repeat the argument with \(\tilde{s} = \alpha_0 + s\) taking the place of \(\frac{s}{2}\) and obtain
    \begin{equation*}
        \norm{N(u,u)}_{H^{\alpha_1}}\lesssim \norm{u}_{H^{\alpha_0 + s}}^{2},  
    \end{equation*}
    for \(\alpha_1 = \min(\alpha_{1,1}, \alpha_{1,2})\), where { 
    \begin{equation*}
        \begin{gathered}
        \alpha_{1,1} = 2\alpha_0 + 2s - \frac{1}{2} - {(r_{i} + r_{j} + r_{k})} - 2\varepsilon_{1} > \alpha_0 + \frac{s}{2}> \alpha_0,\\
        -s-\alpha_0 + r_k - r_i < \alpha_{1,2} \leq \alpha_0 + s - r_j - r_i (>\alpha_0),
        \end{gathered}
    \end{equation*}
    with \(0<\varepsilon_{1} <\frac{\alpha_{0}}{2} + \frac{s}{4}\)}.
    Repeating the procedure, we get in the \(n\)th iteration
    \begin{equation*}
        \norm{N(u,u)}_{H^{\alpha_n}}\lesssim \norm{u}_{H^{\alpha_{n-1} + s}}^2,
    \end{equation*}
    where \(\alpha_n > \alpha_{n-1}\) in each iteration. Continuing like this indefinitely, {and observing that \(\alpha_{n}-\alpha_{n-1}>0\) holds uniformly in \(n\)},  we conclude that \(u\in H^{\infty}(\R)\).
\end{proof}
\section{Estimates for small solutions}
\label{sec:properties} 
By fixing an upper bound \(\mu_0>0\) for \(\mu\), we can estimate the size of the solutions and wave speed. The estimates in this section will depend on \(\mu_0\), but are uniform in \(\mu \in (0,\mu_0)\). We start by showing \cref{thm:main} (i).
\begin{lemma}
    For all \(\mu\in (0,\mu_0)\), a minimizer \(u\) of \(\Gamma_{\mu}\) satisfies
    \begin{equation*}
    \norm{u}_{L^{\infty}}\lesssim \norm{u}_{H^{\frac{s}{2}}}\eqsim \mu^{\frac{1}{2}}.
    \end{equation*}
    \label{lem:size_small}
    \end{lemma}
    \begin{proof}
    We begin by finding a crude upper bound for \(\mathcal{N}(u)\). Recall that \(\mathcal{N}(u)>0\) for any minimizer \(u\) of \(\Gamma_{\mu}\), see \cref{lem:normest}. For some \({\gamma \in (0,2)}\), \cref{lem:B_upper} gives that
    \begin{equation*}
    \mathcal{N}(u) \lesssim \mu^{\frac{3-\gamma}{2}}\norm{u}_{H^{\frac{s}{2}}}^{\gamma}
    \lesssim  \mu^{\frac{3-\gamma}{2}} (\mu + \mathcal{M}(u)-m(0)\mu)^{\frac{\gamma}{2}}
    \lesssim \mu^{\frac{3}{2}} + \mu^{\frac{3-\gamma}{2}}\mathcal{N}(u)^{\frac{\gamma}{2}}.
    \end{equation*}
    {Thus \(\mathcal{N}\) must at least be bounded above by a multiple of one of the terms on the right-hand side, that is either
    \begin{equation*}
        \mathcal{N}(u)\lesssim \mu^{\frac{3}{2}} \quad \text{or} \quad \mathcal{N}(u)\lesssim \mu^{\frac{3-\gamma}{2}}\mathcal{N}(u)^{\frac{\gamma}{2}}.
    \end{equation*}
    Dividing the second equation by \(\mathcal{N}(u)^{\frac{\gamma}{2}}\) and taking the \((\frac{2-\gamma}{2})\)th root, and then recombining the two equations leads to 
    }
    \begin{equation*}
    \mathcal{N}(u) \lesssim \mu^{\frac{3}{2}} + \mu^{\frac{3-\gamma}{2-\gamma}} = \mu^{\frac{3}{2}} + \mu^{1 + \frac{1}{2-\gamma}} \lesssim \mu^{\frac{3}{2}}.
    \end{equation*}
    {In the last estimate, we also used that \(\gamma\in (0,2)\) so that  \(1+\frac{1}{2- \gamma}>\frac{3}{2}\).}
    The estimate for \(\norm{u}_{H^{\frac{s}{2}}}\) now follows:
    \begin{equation*}
    \mu \lesssim \norm{u}_{L^{2}}^{2}\lesssim \norm{u}_{H^{\frac{s}{2}}}^{2}\eqsim \mu + \mathcal{M}(u)= \mu + \mathcal{E}(u)+ \mathcal{N}(u) \lesssim \mu. 
    \end{equation*}  
{To show that \(\norm{u}_{L^{\infty}}\lesssim \norm{u}_{H^{\frac{s}{2}}}\) we will use a similar bootstrap argument as in the proof of \cref{lem:regularity_all}. However, we need to ensure that the estimates are uniform in \(\mu\in (0, \mu_{0})\) and we follow the same approach as in \cite{maehlen2020}. \Cref{eq:solitary_wave} can be rewritten as 
\begin{equation}
    (M-\nu + 1)u= N(u,u) + u. \label{eq:rewrite_small}
\end{equation}
Since \(m(0)-\nu + 1>1\), \((M-\nu + 1)^{-1}\) exists and is a linear operator from \(H^t( \mathbb{R})\) to \(H^{t+s}(\mathbb{R})\). Moreover, the operator norm is uniformly bounded in \(\mu\) since
\begin{equation*}
    \left\|(M-\nu + 1)^{-1} \right\|_{H^{t}\to H^{t+s}} = \sup_{\xi\in \mathbb{R}} \frac{\left\langle \xi \right\rangle^s}{m(\xi) - \nu +1} \leq \sup_{\xi \in \mathbb{R}} \frac{\left\langle \xi \right\rangle^s}{m(\xi) -m(0) + 1} \coloneqq C_{0}.
\end{equation*}
In the proof of \cref{lem:regularity_all}, we showed that there is \(\alpha_{0}>-\frac{s}{2}\) such that \(\left\|N(u,u) \right\|_{H^{\alpha_{0}}}\lesssim \left\|u \right\|_{H^{\frac{s}{2}}}^2\). Since \(\alpha_{0}\leq\frac{s}{2}\) (check \(\alpha_{0,1}\)), we must also have that 
\begin{equation*}
    \left\|N(u,u) + u \right\|_{H^{\alpha_{0}}}\lesssim \left\|u \right\|_{H^{\frac{s}{2}}}^2 + \left\|u \right\|_{H^{\frac{s}{2}}}.
\end{equation*}
Note that the implicit constants are independent of \(\mu\). Analogously to the proof of \cref{lem:regularity_all}, we apply the operator \((M-\nu +1)^{-1}\) to \(N(u,u)+u\), and obtain in the \(n\)-th iteration
\begin{align*}
    &\norm{u}_{H^{\alpha_n + s}}\leq C_{0}\norm{N(u,u)+u}_{H^{\alpha_n}}\lesssim C_{0}(\norm{u}_{H^{\alpha_{n-1}+s}}^2 +\norm{u}_{H^{\alpha_{n-1}+s}}) \\
    &\lesssim \ldots \leq C(C_{0},n)(\norm{N(u,u) + u}_{H^{\alpha_0}}^{2^n} + \norm{N(u,u) + u}_{H^{\alpha_0}})\\
     &\lesssim C(C_{0},n)(\norm{u}_{H^{\frac{s}{2}}}^{2^{n+1}} + \norm{u}_{H^{\frac{s}{2}}}).
\end{align*}
Here, \(C(C_{0},n)\) is a constant depending on \(C_{0}\) and \(n\). Since \(\alpha_{n}-\alpha_{n-1}>0\) uniformly in \(n\), we can find \(m<\infty\) such that \(\alpha_{m}+s> \frac{1}{2}\), so that 
\begin{equation*}
    \left\|u \right\|_{L^{\infty}}\lesssim \left\|u \right\|_{H^{\alpha_{m} + s}} \lesssim \norm{u}_{H^{\frac{s}{2}}}^{2^{m+1}} + \norm{u}_{H^{\frac{s}{2}}},
\end{equation*}
where the implicit constants may depend on \(C_{0}, s, m, r_{i}\), but are independent of \(\mu\). Observing that 
\begin{equation*}
    \norm{u}_{H^{\frac{s}{2}}}^{2^{m+1}} + \norm{u}_{H^{\frac{s}{2}}} \lesssim \norm{u}_{H^{\frac{s}{2}}} 
\end{equation*}
when \(\mu <\mu_{0}\) concludes the proof.}
\end{proof}

To show \cref{thm:main} (ii), we first find an improved upper bound for \(\Gamma_{\mu}\). 
\begin{lemma}
Let \(\beta = \frac{s'}{2s'-1}\). There is a \(\kappa>0\) such that for all \(\mu\in (0,\mu_0)\),
\[
\Gamma_{\mu} <m(0)\mu - \kappa \mu^{1+\beta}.
\]
\label{lem:gamma_upper_small}
\end{lemma}
\begin{proof}
Pick a function \(\phi \in \Sc(\R)\) satisfying \(\mathcal{Q}(\phi) = 1\), \(\hat{\phi}(\xi)\geq 0\) for all \(\xi\in\R\) and \(\hat{\phi}\gtrsim 1\) for \(\abs{\xi}<1\). Define
\[
\phi_{\mu,t} = \sqrt{\mu t} \phi(tx),
\] 
for \(t \in (0,1)\) and observe that \(\mathcal{Q}(\phi_{\mu, t}) = \mu\). In a similar manner as in the proof of \cref{lem:Gamma_bounds} we find that the following holds for some constants \(C_1, C_2 >0\)
\begin{align*}
\mathcal{M}(\phi_{\mu,t}) & \leq m(0)\mu + C_1\mu t^{s'},\\
\mathcal{N}(\phi_{\mu,t}) &\geq C_2 \mu^{\frac{3}{2}}t^{\frac{1}{2}}.
\end{align*}
Since \(\mu\in(0,\mu_0)\), we can set \(t^{s'} = C_3 \mu^{\beta}\) and pick \(C_3\) small enough to guarantee \(t<1\). Combined with the above, then
\begin{align*}
\mathcal{E}(\phi_{\mu,t}) &\leq m(0)\mu + C_1 \mu t^{s'} - {C_2\mu^{\frac{3}{2}}t^{\frac{1}{2}}}\\
&= m(0)\mu - \mu^{1+\beta}(C_2C_3^{\frac{1}{2s'}}\mu^{\frac{1}{2} + \frac{\beta}{2s'} - \beta}-C_1C_3)\\
&=  m(0)\mu - \kappa \mu^{1+\beta}.
\end{align*}
In the last line, we used that \(\frac{1}{2} + \frac{\beta}{2s'} - \beta = 0\) by the definition of \(\beta\) and set \(\kappa = C_2C_3^{1/(2s')} - C_1C_3\).
Since \(\frac{1}{2s'} <1\), we can always pick \(C_3\) small enough to guarantee that \(\kappa >0\) which concludes the proof.
\end{proof}

Inspired by \cite{maehlen2020}, we decompose \(u\) into high- and low-frequency components: Let \(\chi\) be the characteristic function on \([-1,1]\) and set \({\hat{u}_1 = \chi\hat{u},}\) \({\hat{u}_2 =  (1-\chi)\hat{u}}\). 
This decomposition allows us to estimate precisely the size of \(\mathcal{N}(u)\).

\begin{lemma}
Let \(\beta = \frac{s'}{2s'-1}\). For all \(\mu\in (0,\mu_0)\), near minimizers satisfy
\begin{equation*}
\mathcal{\mathcal{N}}(u) \eqsim \mu^{1+\beta}.
\end{equation*}
\label{lem:size_B_small}
\end{lemma}
\begin{proof}
The lower bound for \(\mathcal{N}(u)\) is a direct consequence of \cref{lem:gamma_upper_small}:
\begin{equation*}
\mathcal{N}(u) >\mathcal{M}(u) - m(0)\mu + \kappa \mu^{1+\beta}\gtrsim \mu^{1+\beta}.
\end{equation*}

To show the upper bound, recall that \cref{lem:B_upper} ensures the existence of \(q_i\geq 2\) such that \(H^{\frac{s}{2}}(\R)\hookrightarrow H^{r_i}_{q_i}(\R)\) and 
\begin{equation}
    \mathcal{N}(u) \lesssim \norm{u}_{H^{r_1}_{q_1}}\norm{u}_{H^{r_2}_{q_2}}\norm{u}_{H^{r_3}_{q_3}}.\label{eq:est_u1u2}
\end{equation}
Partitioning \(u= u_1 + u_2\) as described just before this lemma, we wish to find appropriate upper bounds for \(\norm{u_1}_{H^{r_i}_{q_i}}\), \(\norm{u_2}_{H^{r_i}_{q_i}}\) and use these to bound \(\mathcal{N}(u)\). 
Observe first that \(\norm{u_2}_{H^{\frac{s}{2}}}\eqsim \norm{u_2}_{\dot{H}^{\frac{s}{2}}}\). For \(u_1\), we have on the other hand that
\begin{equation*}
    {\norm{u_1}_{\dot{H}^{\frac{s'}{2}}}^2} \eqsim \int_{-1}^1 (m(\xi)-m(0)) \abs{\hat{u}}^2\,d\xi \lesssim \mathcal{M}(u) - m(0)\mu
\end{equation*}

Consider first \(\norm{u_1}_{L^{q_i}}\). By the Gagliardo-Nirenberg interpolation inequality,
\begin{align*}
    \norm{u_1}_{L^{q_i}}&\lesssim \norm{u_1}_{L^2}^{1-\frac{2}{s'}(\frac{1}{2}-\frac{1}{q_i})}\norm{u_1}_{\dot{H}^{\frac{s'}{2}}}^{\frac{2}{s'}(\frac{1}{2}-\frac{1}{q_i})}\\
    &\lesssim \mu^{\frac{1}{2}-\frac{1}{s'}(\frac{1}{2}-\frac{1}{q_i})}(\mathcal{M}(u)-m(0)\mu)^{\frac{1}{s'}(\frac{1}{2}-\frac{1}{q_i})}\\
    &\lesssim \mu^{\frac{1}{2}-\frac{1}{s'}(\frac{1}{2}-\frac{1}{q_i})}\mathcal{N}(u)^{\frac{1}{s'}(\frac{1}{2}-\frac{1}{q_i})}.
\end{align*}
If \(r_i \leq 0\), then \(\norm{u_1}_{H^{r_i}_{q_i}}\leq \norm{u_1}_{L^{q_i}}\), whereas if \(r_i\geq 0\), then
\begin{align*}
    \norm{u_1}_{\dot{H}^{r_i}_{q_i}}&\lesssim \norm{u_1}_{L^2}^{1-\frac{2}{s'}(\frac{1}{2}-\frac{1}{q_i} + r_i)}\norm{u_1}_{\dot{H}^{\frac{s'}{2}}}^{\frac{2}{s'}(\frac{1}{2}-\frac{1}{q_i} + r_i)}\\
    &\lesssim \mu^{\frac{1}{2}-\frac{1}{s'}(\frac{1}{2}-\frac{1}{q_i}+ r_i)}(\mathcal{M}(u)-m(0)\mu)^{\frac{1}{s'}(\frac{1}{2}-\frac{1}{q_i}+ r_i)}\\
    &\lesssim \mu^{\frac{1}{2}-\frac{1}{s'}(\frac{1}{2}-\frac{1}{q_i})}\mathcal{N}(u)^{\frac{1}{s'}(\frac{1}{2}-\frac{1}{q_i})}.
\end{align*}
In the last line we used that \(\mathcal{N}(u)\lesssim \mu^{\frac{3}{2}}\lesssim\mu\). In both cases,
\begin{equation*}
    \norm{u_1}_{H^{r_i}_{q_i}}\lesssim \norm{u_1}_{L^{q_i}} + \norm{u_1}_{\dot{H}^{r_i}_{q_i}} \lesssim \mu^{\frac{1}{2}-\frac{1}{s'}(\frac{1}{2}-\frac{1}{q_i})}\mathcal{N}(u)^{\frac{1}{s'}(\frac{1}{2}-\frac{1}{q_i})}.
\end{equation*}

Now consider \(\norm{u_2}_{H^{r_i}_{q_i}}\).  Since \(q_i\) was chosen to ensure that \({\norm{u}_{H^{r_i}_{q_i}}\lesssim \norm{u}_{H^{\frac{s}{2}}}}\) and \(\norm{u}_{H^{\frac{s}{2}}}\eqsim \mu^{1/2}\),
\begin{equation*}
    \norm{u_2}_{H^{r_i}_{q_i}}\!\lesssim \norm{u_2}_{H^{\frac{s}{2}}}\!\lesssim { \norm{u}_{H^{\frac{s}{2}}}^{1 - \frac{2}{s'}(\frac{1}{2}-\frac{1}{q_i})}\norm{u}_{\dot{H}^{\frac{s}{2}}}^{\frac{2}{s'}(\frac{1}{2}-\frac{1}{q_i})}}\!\lesssim \mu^{\frac{1}{2}-\frac{1}{s'}(\frac{1}{2}-\frac{1}{q_i})}\mathcal{N}(u)^{\frac{1}{s'}(\frac{1}{2}-\frac{1}{q_i})}.
\end{equation*}
Inserting these two estimates for \(i=1,2,3\) into \cref{eq:est_u1u2} and keeping in mind that \({\frac{1}{q_1} + \frac{1}{q_2} + \frac{1}{q_3} = 1}\), we find that
\begin{align*}
    \mathcal{N}(u) &\lesssim \norm{u}_{H^{r_1}_{q_1}}\norm{u}_{H^{r_2}_{q_2}}\norm{u}_{H^{r_3}_{q_3}}\\
    &\lesssim \Pi_{i=1,2,3}\,\mu^{\frac{1}{2}-\frac{1}{s'}(\frac{1}{2}-\frac{1}{q_i})}\mathcal{N}(u)^{\frac{1}{s'}(\frac{1}{2}-\frac{1}{q_i})}\\
    &\lesssim \mu^{\frac{3}{2}-\frac{1}{2s'}}\mathcal{N}(u)^{\frac{1}{2s'}}.
\end{align*}
Dividing both sides by \(\mathcal{N}(u)^{\frac{1}{2s'}}\) and taking the \(1 - \frac{1}{2s'}\) root gives the desired estimate: 
\begin{equation*}
    \mathcal{N}(u)\lesssim \mu^{\frac{\frac{3}{2}- \frac{1}{2s'}}{1-\frac{1}{2s'}}} = \mu^{\frac{3s'-1}{2s'-1}} = \mu^{1 + \beta}.
\end{equation*}
\end{proof}

With \cref{lem:size_B_small} in hand, we can finally show \cref{thm:main} (ii).
\begin{lemma}
If \(\mu\in (0,\mu_0)\), then any minimizer of \(\Gamma_{\mu}\) solves \cref{eq:solitary_wave} with wave speed \(\nu\) satisfying 
\begin{equation*}
m(0) - \nu \eqsim  \mu^{\beta},\,\,\beta = \frac{s'}{2s'-1}.\label{eq:wave_speed_small}
\end{equation*}\label{lem:wave_speed_small}
\end{lemma}
\begin{proof}
As in the proof of \cref{lem:wave_speed_all}, a minimizer satisfies 
\begin{equation*}
    -2\nu \mu + 2\mathcal{M}(u)- 3\mathcal{N}(u) = 0,
\end{equation*}
where \(\nu\) is the wave speed and depends on \(\mu\). By \cref{lem:normest} \(\mathcal{N}(u)>0\). Combined with the estimated size of \(\mathcal{N}(u)\) in {\cref{lem:size_B_small}}, then
\begin{equation*}
    \begin{split}
\nu &= \frac{\mathcal{E}(u) - \frac{1}{2}\mathcal{N}(u)}{\mu}\\
&\leq m(0) - C_1\mu^{\beta},
    \end{split}
\end{equation*}
for some \(C_1 >0\). On the other hand, 
\begin{equation*}
    \begin{split}
\nu &= \frac{\mathcal{M}(u) - \frac{3}{2}\mathcal{N}(u)}{\mu}\\
&\geq m(0) -  C_2 \mu^{\beta}, 
    \end{split}
\end{equation*}
for some \(C_2>0\). 
\end{proof}

\section{The modified assumption \labelcref{assump:Nstar}}
\label{sec:Bstar}
In this section, we assume that \(N\) satisfies the modified assumption \labelcref{assump:Nstar} instead of \cref{assump:N} and show that \cref{thm:main} still holds. This allows for pseudo-products that cannot nessecarily be reduced to a combination of linear Fourier operators.

A bilinear counterpart to Hörmander--Mikhlin's boundedness criterion for linear operators is Coifman--Meyer's result \cite{coifman1978}, which requires the symbol \(\sigma(\xi, \eta)\) to satisfy 
\begin{equation*}
    \abs{\partial^{\alpha}\sigma(\xi, \eta)}\lesssim \jpb{\abs{\xi} + \abs{\eta}}^{-\abs{\alpha}},
\end{equation*}
for sufficiently many \(\alpha\). Because of the necessary symmetry in \(-\xi, \xi-\eta\) and \(\eta\) for a variational formulation, we instead rely on Corollary 1.4 from \cite{grafakos2020}.
\begin{proposition}[\!\!{\cite[Corollary 1.4]{grafakos2020}}]
    Let \({q \in [1,4)}\) and let \({M_q = \left\lfloor \frac{2}{4-q}\right\rfloor +1}\). Assume that \(m\in L^q(\R^2)\cap C^{M_q}(R^2)\) satisfies
    \begin{equation}
        \norm{\partial^{\alpha}m}_{L^{\infty}} \leq C_0 <\infty \text{ for } \abs{\alpha}\leq M_q.\label{eq:Grafakosreq}
    \end{equation}
    Suppose \(p_1,p_2,p\) satisfy \(2\leq p_1, p_2\leq \infty, 1\leq p\leq 2\) and \(\frac{1}{p}= \frac{1}{p_1} + \frac{1}{p_2}\). Then the bilinear operator \(T_m\) with symbol \(m\) is a bounded operator from \(L^{p_1}(\R)\times L^{p_2}(\R)\) to \(L^p(\R)\). 
    \label{prop:Grafakos}
\end{proposition}

Under \cref{assump:Nstar}, \(n\) can be written as
\begin{equation*}
    n(\xi-\eta, \eta) = \sum_{cyc} f(\xi, \xi-\eta, \eta).
\end{equation*}
Introducing a symbol
\begin{equation*}
    \tilde{n}(\xi-\eta, \eta) = \frac{f(\xi, \xi-\eta, \eta)}{\jpb{\xi}^{\tilde{r}_1}\jpb{\xi-\eta}^{\tilde{r}_2}\jpb{\eta}^{\tilde{r}_3}},
\end{equation*}
for fitting exponents \(\tilde{r}_i\), we can rewrite \(\mathcal{N}\) as
\begin{equation}
    \mathcal{N}(u) = \intR \Lambda^{\tilde{r}_1}u \tilde{N}(\Lambda^{\tilde{r}_2}u, \Lambda^{\tilde{r}_3}u)\,dx\label{eq:funcform_Bstar},
\end{equation}
where \(\tilde{N}\) is the bilinear Fourier multiplier with symbol \(\tilde{n}\). Using \cref{prop:Grafakos}, we show that \(\tilde{r}_i\) can be chosen to satisfy the requirements of the original assumption \labelcref{assump:N}, while at the same time guaranteeing that \(\tilde{N}\) is bounded. This will allow us to reduce many of the arguments under the modified assumption to the case already covered.

\begin{lemma}
    Let \(p_1, p_2, p\) be as in \cref{prop:Grafakos}. Suppose \(n\) satisfies \cref{assump:Nstar}. Then there are \(\tilde{r}_1,\tilde{r}_2,\tilde{r}_3\) satisfying the requirements of \cref{assump:N}, that is \(\tilde{r}_i<\frac{s}{2}\),
    \begin{equation*}
        \begin{split}
        \sum_{i=1}^3 \tilde{r}_i &< s-\frac{1}{2},\\
        \tilde{r}_i +\tilde{r}_j &<\frac{3s}{2}-\frac{1}{2} \text{ for } i\neq j,
        \end{split}
    \end{equation*}
    such that the bilinear Fourier multiplier \(\tilde{N}\) with symbol \(\tilde{n}\) is bounded from \(L^{p_1}(\R)\times L^{p_2}(\R)\) to \(L^p(\R)\).
    \label{lem:Bstar_toB}
\end{lemma}
\begin{proof}
    Pick \(\tilde{r}_1, \tilde{r}_2, \tilde{r}_3\) such that  \(r_i + \frac{1}{4} <\tilde{r}_i < \frac{s}{2}\). This is possible since by assumption \(r_i <\frac{s}{2}-\frac{1}{4}\). By picking each \(\tilde{r}_i\) sufficiently small, i.e. setting \(\tilde{r}_i = r_i + \frac{1}{4} + \varepsilon\) for \(\varepsilon\) sufficiently small, we can also ensure that
    \begin{equation*}
        \sum_{i=1}^3 \tilde{r}_i = \sum_{i=1}^3 (r_i + \frac{1}{4} + \varepsilon) < s-\frac{5}{4} + \frac{3}{4} = s- \frac{1}{2}, 
    \end{equation*}
    and for \(i\neq j\) that
    \begin{equation*}
        \tilde{r}_i + \tilde{r}_j = r_i + r_j + \frac{1}{2} + 2\varepsilon <\frac{3s}{2} - 1 + \frac{1}{2} = \frac{3s}{2} -\frac{1}{2}.
    \end{equation*}
    {Both these inequalities hold as long as \( \varepsilon\) satisfies 
    \[0<\varepsilon<\min\left(\frac{1}{3}\left(s-\frac{5}{4} - \sum r_{i}\right), \frac{1}{2}\left(\frac{3s}{2} - 1 - (r_{i} + r_{j})\right)\right),\] 
    possible since both expressions on the right hand side are strictly positive.}
    Set \(a_i = \tilde{r}_i - r_i, i = 1,2,3\). Then
    \begin{equation*}
        \abs{\tilde{n}(\xi-\eta, \eta)} \lesssim \frac{1}{\jpb{\xi}^{a_1}\jpb{\xi-\eta}^{a_2}\jpb{\eta}^{a_3}},
    \end{equation*}
    and all derivatives \(\tilde{n}\) are bounded due to the growth restriction on the derivatives of \(f\), see \cref{assump:Nstar}. Furthermore, 
    \begin{align*}
        \jpb{\abs{\xi}+\abs{\eta}}^{\min_{i\neq j}(a_i + a_j)} &\lesssim \jpb{\xi}^{a_1 + \min(a_2, a_3)} +  \jpb{\eta}^{a_3 + \min(a_1,a_2)}\\
        &\lesssim \jpb{\xi}^{a_1}\jpb{\xi-\eta}^{a_2}\jpb{\eta}^{a_3}.
    \end{align*}
    Since \(a_i = \tilde{r}_i - r_i >\frac{1}{4}\), then \(\min_{i\neq j}(a_i + a_j)>\frac{1}{2}\) so that
    \begin{equation*}
        \frac{1}{\jpb{\xi}^{a_1}\jpb{\xi-\eta}^{a_2}\jpb{\eta}^{a_3}} \lesssim \frac{1}{\jpb{\abs{\xi}+\abs{\eta}}^{\alpha}},
    \end{equation*}
    for some \(\alpha >\frac{1}{2}\). 

    Now we apply \cref{prop:Grafakos}. The symbol \(\tilde{n}\) clearly satisfies \(\tilde{n}\in C^{\infty}(\R^2)\) and  \cref{eq:Grafakosreq} for any \(M_q\). Furthermore, \(\tilde{n}\in L^q(\R^2)\) for any \(4>q>\frac{2}{\alpha}\) by comparison with 
    \begin{equation*}
        \frac{1}{\jpb{\abs{\xi}+\abs{\eta}}^{\alpha}} \in L^{q}(\R^2).
    \end{equation*}
    Hence \(\tilde{N}\) is bounded from \(L^{p_1}(\R)\times L^{p_2}(\R)\) to \(L^p(\R)\) for \(p_1, p_2, p\) as in \cref{prop:Grafakos}. 
\end{proof}

To prove \cref{thm:main} under the modified assumption \labelcref{assump:Nstar}, we show that each lemma that holds under \cref{assump:N} also holds under \cref{assump:Nstar}. Inspecting the proofs in Sections \ref{sec:functionals} to \ref{sec:properties}, the lemmas of the preceding sections can be divided into three categories: First, those where the proof relies only on properties of \(M\) or results established in previous lemmas. These require no modification. Second, the lemmas where the proof relies on the estimate
\begin{equation}
    \abs{\mathcal{N}(u)}\lesssim \norm{u}_{H^{\tilde{r}_1}_{q_1}}\norm{u}_{H^{\tilde{r}_2}_{q_2}}\norm{u}_{H^{\tilde{r}_3}_{q_3}},\label{eq:qi_ri}
\end{equation}
for appropirate \(q_i\) and where \(\tilde{r}_i\) are in accordance with \cref{assump:N}. This applies to Lemmas \ref{lem:B_upper}, \ref{lem:normest}, \ref{lem:minimizer} and \ref{lem:size_B_small}. The third category are lemmas where the proof relies explicitly on the form of the symbol \(n\) beyond \cref{eq:qi_ri}. This applies to part of Lemmas \ref{lem:Gamma_bounds}, \ref{lem:commutator}, and \ref{lem:regularity_all} and we deal with them separately. 

From the formulation \labelcref{eq:funcform_Bstar} and \cref{lem:Bstar_toB}, we immediately get \cref{eq:qi_ri}.

\begin{lemma}[Upper bound for \(\mathcal{N}(u)\)]
    Let \(u\in H^{\frac{s}{2}}(\R)\). Suppose that \(q_i, {i = 1,2,3}\), satisfy
    \begin{equation*}
        q_1\geq 2 \quad\text{and}\quad \sum_{i=1}^3 \frac{1}{q_i} = 1
    \end{equation*}
    Then there are \(r_i, i = 1,2,3,\) satisfying the requirements of \cref{assump:N}
    such that
    \begin{equation*}
        \abs{\mathcal{N}(u)}\lesssim \norm{u}_{H^{\tilde{r}_1}_{q_1}}\norm{u}_{H^{\tilde{r}_2}_{q_2}}\norm{u}_{H^{\tilde{r}_3}_{q_3}}.
    \end{equation*}
    \label{lem:qi_ri}
\end{lemma}
\begin{proof}
    Pick \(\tilde{r}_1, \tilde{r}_2, \tilde{r}_3\) according to \cref{lem:Bstar_toB}, and let \(\tilde{N}\) be as in \cref{eq:funcform_Bstar}. Then
    \begin{align*}
        \abs{\mathcal{N}(u)} &\lesssim\int_{\R}\abs{\Lambda^{\tilde{r}_1}u\tilde{N}(\Lambda^{\tilde{r}_2}u, \Lambda^{\tilde{r}_3}u)}dx\\
        &\lesssim \norm{\Lambda^{\tilde{r}_1}u}_{L^{q_1}} \norm{\tilde{N}(\Lambda^{\tilde{r}_2}u, \Lambda^{\tilde{r}_3}u)}_{L^{\tilde{q}}}, 
    \end{align*}
    where \(\frac{1}{\tilde{q}} = \frac{1}{q_2} + \frac{1}{q_3}\). Since \(q_1\geq 2\), then \(\tilde{q}\leq 2\) and we apply \cref{lem:Bstar_toB}
    \begin{align*}
        \abs{\mathcal{N}(u)} &\lesssim \norm{\Lambda^{\tilde{r}_1}u}_{L^{q_1}} \norm{\tilde{N}(\Lambda^{\tilde{r}_2}u, \Lambda^{\tilde{r}_3}u)}_{L^{\tilde{q}}}\\
        &\lesssim \norm{\Lambda^{\tilde{r}_1}u}_{L^{q_1}} \norm{\Lambda^{\tilde{r}_2}u}_{L^{q_2}}\norm{\Lambda^{\tilde{r}_3}u}_{L^{q_3}}\\
        &\lesssim \norm{u}_{H^{\tilde{r}_1}_{q_1}} \norm{u}_{H^{\tilde{r}_2}_{q_2}}\norm{u}_{H^{\tilde{r}_3}_{q_3}}.
    \end{align*}
\end{proof}
Observe that we have put no further assumptions on \(q_1,q_2, q_3\) other than those already covered by the requirements in \cref{lem:B_upper}. This means that all variations of the argument in \cref{lem:B_upper} can be repeated in exactly the same manner to arrive at the results of \cref{lem:B_upper,lem:normest,lem:minimizer,lem:size_B_small}.

We now turn to the remaining three lemmas that require modification. In the proof of \cref{lem:Gamma_bounds}, the properties of \(N\) are used to find that \({\mathcal{N}(u)\gtrsim t^{1/2}}\). However, the only properties used are that \(n\) is continuous and strictly positive at the origin, and growth bounds on the absolute value of \(n\). As these properties still hold under \cref{assump:Nstar}, so does the result of the lemma. Proving an equivalent of \cref{lem:commutator} requires more changes but the argument is easier due to the stricter growth bounds on \(n\). 

\begin{lemma}
    Let \(u\in H^{\frac{s}{2}}(\R)\). Let \(\rho\) be a non-negative Schwartz function, and let \({\rho_R(x) = \rho(x/R)}\). Assume \(N\) satisfies \cref{assump:Nstar}. Then 
    \begin{equation*}
        \begin{split}
        \lim_{\R\to\infty}\left|\int_{\R}v(\rho_{R}N(u,u) - N(\rho_{R}u, u))\,dx\right| = 0\\
        \lim_{\R\to\infty}\left|\int_{\R}v((1-\rho_{R})N(u,u) - N((1-\rho_{R})u, u))\,dx\right| = 0
        \end{split}
    \end{equation*}
\end{lemma}
\begin{proof}
We assume without loss of generality that \(r_1\leq r_2\leq r_3\). Combining Plancherel's and Fubini's theorems, we have that
\begin{equation}
    \begin{split}
    |\int_{\R}&v(\rho_{R}N(u,u) - N(\rho_{R}u, u))\,dx| \\
    &= |\int_{\R} \overline{\hat{v}(\xi)}(\hat{\rho}_{R}\ast \int_{\R}n(\cdot-\eta,\eta)\hat{u}(\cdot-\eta)\hat{u}(\eta)\,d\eta) (\xi)\\
    &\quad- \int_{\R}n(\xi-\eta, \eta)(\hat{\rho}_{R}\ast \hat{u})(\xi-\eta) \hat{u}(\eta) \,d\eta\,d\xi|\\
&\leq \int_{\R^{3}}|\overline{\hat{v}(\xi)}\hat{\rho}_{R}(t)\hat{u}(\eta)\hat{u}(\xi-t-\eta)||{n(\xi-t-\eta, \eta) - n(\xi-\eta, \eta)}|\,d\eta\,dt\,d\xi.
    \end{split}
    \label{eq:com_T_est2}
\end{equation}
By the mean value theorem, 
\begin{align*}
    |{n(\xi-t-\eta, \eta) - n(\xi-\eta, \eta)}|&\lesssim \abs{t}\jpb{t}^{r_3}\jpb{\xi}^{r_3}\jpb{\eta}^{r_3}\jpb{\xi-t-\eta}^{r_3},\\
\intertext{which, due to \(r_3<\frac{s}{2}-\frac{1}{4}\), leads to}
    |{n(\xi-t-\eta, \eta) - n(\xi-\eta, \eta)}| &\lesssim \abs{t}\jpb{t}^{\frac{s}{2}}\jpb{\xi}^{r_3-\frac{1}{4}}\jpb{\eta}^{\frac{s}{2}}\jpb{\xi-t-\eta}^{\frac{s}{2}}, \label{eq:mvt_Bstar}
\end{align*}
We insert this into \cref{eq:com_T_est2},
\begin{align*}
&|\int_{\R}v(\rho_{R}N(u,u) - N(\rho_{R}u, u))\,dx|\\
&\lesssim \int_{\R}|\hat{\rho}_{R}(t)|t|\jpb{t}^{\frac{s}{2}}|\int_{\R}|\overline{\hat{v}(\xi)}\jpb{\xi}^{r_3-\frac{1}{4}}|(|\hat{u}(\cdot)\jpb{\cdot}^{\frac{s}{2}}|\ast|\hat{u}(\cdot-t)\jpb{\cdot-t}^{\frac{s}{2}}|)(\xi)\,d\xi\,dt,\\
\intertext{and apply Young's convolution inequality to the inner integral, then H\"older's inequality, observing that \(r_3 - \frac{1}{4}-\frac{s}{2}<-\frac{1}{2}\),}
&\leq \norm{\overline{\hat{v}}\jpb{\cdot}^{\frac{s}{2}}}_{L^{2}}\norm{\jpb{\cdot}^{r_3-\frac{1}{4}-\frac{s}{2}}}_{L^2}\norm{\hat{u}\jpb{\cdot}^{s/2}}_{L^{2}}\norm{\hat{u}\jpb{\cdot}^{\frac{s}{2}}}_{L^{2}}\int_{\R}|\hat{\rho}(Rt)|t|\jpb{t}^{\frac{s}{2}}\,dt\\
&\leq \norm{v}_{H^{\frac{s}{2}}}\norm{u}_{H^{s/2}}^2\frac{1}{|R|}\int_{\R}|\hat{\rho}(t)|t|\jpb{t/R}^{r/2}\,dt.
\end{align*}
This expression approaches zero as \(R\) tends to infinity since \(\rho\in \Sc(\R)\), showing the first part of the lemma. As in \cref{lem:commutator} the second part follows from a simple calculation using that \(N\) is a bilinear operator.
\end{proof}

Finally, we prove \cref{lem:regularity_all} under \cref{assump:Nstar}. Here, the proof also simplifies due to the stricter growth bounds. 
\begin{lemma}[Regularity of solutions]
    Any solution \(u\in H^{\frac{s}{2}}(\R)\) of \cref{eq:solitary_wave} with \(\mathcal{Q}(u) = \mu\) is also in \(H^{\infty}(\R)\).
\end{lemma}
\begin{proof}
    As in \cref{lem:regularity_all} we can write
    \begin{equation}
    (M-\nu)u = N(u,u),
    \label{eq:reg_start2}
    \end{equation}
    and observe that \((M-\nu)^{-1}\colon H^{t}(\R)\to H^{t + s}(\R)\). Assume without loss of generality that \(r_1 \leq r_2 \leq r_3\). Pick \(\alpha_0\) with \({-\frac{s}{2}<\alpha_0<\frac{s}{2} - \frac{1}{2} - (r_2 + r_3)}\), always possible since \(r_2 + r_3 <s-\frac{1}{2}\) by assumption. We will show that \(N(u,u)\in H^{\alpha_0}(\R)\). { Observe that
    \begin{equation*}
        \norm{N(u,u)}_{H^{\alpha_0}} \lesssim \|\left\langle \xi \right\rangle^{\alpha_{0}} \sum_{cyc}\left\langle \xi \right\rangle^{r_{1}}  (\left\langle \cdot \right\rangle^{r_{2}} \left|\hat{u}\right|\ast \left\langle \cdot \right\rangle^{r_{3}} \left|\hat{u}\right|) (\xi) \|_{L^2},
    \end{equation*}
    where the cyclic sum is in \(r_{1}, r_{2}, r_{3}\). H\"older's and Young's inequalities imply 
    \begin{align*}
        \left\|\left\langle \xi \right\rangle^a (\left\langle \cdot \right\rangle^b \hat{u}\ast \left\langle \cdot \right\rangle^c \hat{u}) (\xi) \right\|_{L^2}\lesssim \left\|\left\langle \xi \right\rangle^a  \right\|_{L^1}      \left\|\left\langle \xi \right\rangle^b \hat{u}(\xi) \right\|_{L^2}   \left\|\left\langle \xi \right\rangle^c \hat{u}(\xi) \right\|_{L^2},
    \end{align*}
    which is bounded if \(b,c \leq \frac{s}{2}\) and \(a<-\frac{1}{2}\). To bound \(\norm{N(u,u)}_{H^{\alpha_0}}\), we therefore want to redistribute the powers of \(\left\langle \xi \right\rangle, \left\langle \xi-\eta \right\rangle, \left\langle \eta \right\rangle\) using that \({\left\langle \xi \right\rangle\lesssim \left\langle \xi-\eta \right\rangle \left\langle \eta \right\rangle}\).
    Let \(\theta_1 = \theta_2 = \frac{s}{2}-r_3, \theta_3 = \frac{s}{2}-r_2\). Then
    \begin{align*}
        \norm{N(u,u)}_{H^{\alpha_0}} &\lesssim\|\left\langle \xi \right\rangle^{\alpha_{0}+r_{1}- \theta_{1}} (\left\langle \cdot \right\rangle^{r_{2}+\theta_{1}} \hat{u}\ast \left\langle \cdot \right\rangle^{r_{3}+\theta_{1}} \hat{u}) (\xi) \|_{L^2}\\
        &\quad +\|\left\langle \xi \right\rangle^{\alpha_{0}+r_{2}- \theta_{2}} (\left\langle \cdot \right\rangle^{r_{3}+\theta_{2}} \hat{u}\ast \left\langle \cdot \right\rangle^{r_{1}+\theta_{2}} \hat{u}) (\xi) \|_{L^2}\\
        &\quad +\|\left\langle \xi \right\rangle^{\alpha_{0}+r_{3}- \theta_{3}} (\left\langle \cdot \right\rangle^{r_{1}+\theta_{3}} \hat{u}\ast \left\langle \cdot \right\rangle^{r_{2}+\theta_{3}} \hat{u}) (\xi) \|_{L^2}\\
        &\lesssim \norm{u}_{H^{\frac{s}{2}}}^2 {\norm{\jpb{\xi}^{\alpha_0 + r_2 + r_3 - \frac{s}{2}}}_{L^1}}. 
    \end{align*}}
     By choice of \(\alpha_0\), then
    \begin{equation*}
        \norm{N(u,u)}_{H^{\alpha_0}}\lesssim \norm{u}_{H^{\frac{s}{2}}}^2.
    \end{equation*}  
    Combined with \cref{eq:reg_start2}, this implies that \((M-\nu)u \in H^{\alpha_0}(\R)\) which in turn implies \(u\in H^{\alpha_0 + s}(\R)\). Due to our choice of \(\alpha_0\), we have that \(\alpha_0 + s\) is strictly bigger than \(\frac{s}{2}\). Now we can do the same with \(\tilde{s} = \alpha_0 + s\) taking the place of \(\frac{s}{2}\) and obtain that
    \begin{equation*}
        \norm{N(u,u)}_{H^{\alpha_1}}\lesssim \norm{u}_{H^{\alpha_0 + s}}^{2},  
    \end{equation*}
    for \(\alpha_1 = \frac{s}{2} + 2\alpha_0 > \alpha_0\). The argument is the same as above, except now we set \(\theta_1 = \theta_2 = \alpha_0 + s - r_3, \theta_3 = \alpha_0 + s - r_2\). Iterating \(n\) times, 
    \begin{equation*}
        \norm{N(u,u)}_{H^{\alpha_n}}\lesssim \norm{u}_{H^{\alpha_{n-1} + s}}^2,
    \end{equation*}
    where \(\alpha_n + s = \frac{s}{2} + (n+1)(\frac{s}{2} + \alpha_0)\). We remark that the implicit constants increase with increasing \(\alpha_{n-1}\) due to the estimate \(\jpb{\xi}^{\theta_i}\lesssim \jpb{\xi-\eta}^{\theta_i}\jpb{\eta}^{\theta_i}\). Recall that \(\frac{s}{2} + \alpha_0\) is strictly positice by our choice of \(\alpha_0\). We can continue indefinitely and conclude that \(u\in H^{\infty}(\R)\).
\end{proof}
We have shown that all intermediate results hold if we assume \(N\) satisfies \cref{assump:Nstar}, thus proving \cref{thm:main} also in this case.

\section*{Acknowledgments}
The author would like to thank the referees for their careful reading, and their many valuable comments and helpful suggestions.

\appendix
\section{Necessity of symmetry assumption for a variational formulation}
\label{sec:symmetry}
Let \(N\) denote a bilinear Fourier multiplier with symbol \(n\) and \(P\) a trilinear Fourier multiplier with symbol \(p\):
\begin{equation*}
    \widehat{P(u,v,w)}= \intR p(\xi-\eta, \eta-\sigma, \sigma) \hat{u}(\xi-\eta)\hat{u}(\eta - \sigma) \hat{u}(\sigma)\,d\sigma\,d\eta.
\end{equation*}
Define the functional \(\mathcal{P}(u)\) by
\begin{equation*}
    \mathcal{P}(u) = \intR P(u,u,u)\,dx.
\end{equation*}
In this section, we derive sufficient and necessary conditions on \(N, P\) such that the Fréchet derivative 
\begin{equation*}
    D\mathcal{P}[u](v) = \intR vN(u,u) \,dx.
\end{equation*}
We will need the useful identities below, which hold as long as \(f,g,h\) are smooth enough for all integrals to make sense. 
\begin{lemma}
    Let \(f, g, h \in \Sc(\R)\) and \(p\) bounded above by a polynomial. Let \(P, P', P''\) be the trilinear Fourier multipliers with symbols \(p, p',p''\) respectively, where
    \begin{equation*}
        \begin{split}
    p'(\xi_1, \xi_2, \xi_3) &= p(\xi_3, \xi_1, \xi_2),\\
    p''(\xi_1, \xi_2, \xi_3) &= p(\xi_2, \xi_3, \xi_1).
        \end{split}
    \end{equation*}
    Then
    \begin{equation*}
        P(f,g,h) = P'(g,h, f) = P''(h, f, g),
    \end{equation*}  
    \label{prop_trilinear}
    \end{lemma}
    \begin{proof}
    The first equality can be shown using the substitution \(\sigma \mapsto \eta -\sigma\) and \(\eta \mapsto \xi -\sigma\) combined with Plancherel and Fubini's theorems.
    \begin{align*}
    \widehat{P(f,g,h)}(\xi) &= \int_{\R^{2}} p(\xi-\eta, \eta-\sigma, \sigma) \hat{f}(\xi-\eta) \hat{g}(\eta-\sigma)\hat{h}(\sigma)\,d\sigma \,d\eta\\
    &= \int_{\R^{2}} p(\sigma, \xi-\eta, \eta-\sigma)\hat{f}(\sigma)\hat{g}(\xi-\eta) \hat{h}(\eta-\sigma) \,d\sigma\,d\eta\\
    &= \int_{\R^{2}} p'(\xi-\eta, \eta - \sigma,\sigma)\hat{g}(\xi-\eta)\hat{h}(\eta-\sigma)\hat{f}(\sigma)\,d\sigma\,d\eta\\
    &= \widehat{P'(g,h,f)}(\xi){.}
    \end{align*}
    To show the second inequality, observe that
    \begin{equation*}
        (p')'(\xi_1, \xi_2, \xi_3)= p'(\xi_3, \xi_1, \xi_2) = p(\xi_2, \xi_3, \xi_1) = p''(\xi_1, \xi_2, \xi_3), 
    \end{equation*}
    so that 
    \begin{equation*}
        P''(h,f,g) = (P')'(h,f,g) = P'(g, h, f) = P(f,g,h). 
    \end{equation*}
    \end{proof}

We assume that
\begin{equation}
    \begin{gathered}
        p(\xi-\eta, \eta-\sigma, \sigma)\lesssim \jpb{\xi-\eta}^{t}\jpb{\eta-\sigma}^{t}\jpb{\sigma}^{t},\\
        n(\xi-\eta, \eta)\lesssim\jpb{\xi}^{t}\jpb{\xi-\eta}^{t}\jpb{\eta}^{t} 
    \end{gathered}
    \label{eq:symgrowth}
\end{equation}
for some \(t>0\). 

\begin{proposition}
Let \(N, P, \mathcal{P}\) be as described, with \(p,n\) satisfying \cref{eq:symgrowth}. Let \(u, v \in H^{s}(\R)\), where \(s>t+\frac{1}{2}\). Then
\[
D\mathcal{P}[u](v) = \int v N(u,u)\,dx 
\]
if and only if
\[
n(\xi-\eta, \eta)= p(\eta, \xi-\eta, -\xi) + p(-\xi, \eta,\xi - \eta)+ p(\xi-\eta, -\xi, \eta).
\]
In this case, 
\begin{equation*}
    \mathcal{P}(u) = \int_{\R} P(u,u,u) \,dx = \frac{1}{3}\int_{\R}uN(u,u)\,dx.
\end{equation*}
\label{opposite_dir}
\end{proposition}

\begin{proof}
Since the Fréchet derivative is unique,
\[
D\mathcal{P}[u](v) = \intR vN(u,u) \,dx
\]
if and only if
\begin{equation}
\frac{|\intR 1\cdot \F^{-1}(\widehat{P(u+v,u+v,u+v)} - \widehat{P(u,u,u)})\,dx - \intR vN(u,u)\,dx|}{\norm{v}_{H^{s}}} \to 0
\label{derivative_start}
\end{equation}
as \(\norm{v}_{H^{s}}\to 0\).

We apply Lemma \ref{prop_trilinear} and find that
\begin{align*}
P(u+v,u+v&,u+v) -P(u,u,u)\\
&= P''(v,u,u) + P'(v,u,u) + P'(v,v,u)\\
&\quad+ P(v,u,u) + P''(v,v, u) + P(v,v,u) + P(v,v,v).  
\end{align*}
When inserted back into the integral above, all terms where \(v\) appears at least twice will vanish as \(\norm{v}_{H^{s}}\to 0\), as demonstrated for \(P(v,v,u)\):
\begin{align*}
    \frac{\abs{\intR 1\cdot \mathcal{F}^{-1}(\widehat{P(v,v,u)})\,dx}}{\norm{v}_{H^{s}}}
    &=\frac{|\int_{\R^{2}} p(-\eta, \eta-\sigma, \sigma)\hat{v}(-\eta)\hat{v}(\eta-\sigma)\hat{u}(\sigma)\,d\sigma\,d\eta|}{\norm{v}_{H^{s}}}\\
    &\lesssim \frac{\norm{|\jpb{\cdot}^{t}\hat{v}| (|\jpb{\cdot}^{t}\hat{u}| \ast |\jpb{\cdot}^{t}\hat{u}|)}_{L^{1}}}{\norm{v}_{H^{s}}}\\
    &\lesssim \frac{\norm{v}_{H^{t}}^2 \norm{|\jpb{\cdot}^{t}\hat{u}|}_{L^{1}}}{\norm{v}_{H^{s}}}\to 0
    \end{align*}
as \(\norm{v}_{H^{s}}\to 0\) since \(\norm{|\jpb{\cdot}^{t}\hat{u}|}_{L^{1}}\lesssim \norm{u}_{H^s}\norm{\jpb{\cdot}^{t-s}}_{L^2}\) is finite. 
\Cref{derivative_start} then becomes 
\begin{align*}
&\frac{1}{\norm{v}_{H^{s}}}|\int_{\R} 1\cdot \F^{-1}(\widehat{P''(v,u,u)} + \widehat{P'(v,u,u)} + \widehat{P(v,u,u)})\,dx - \int_{\R} vN(u,u)\,dx|\\
&\qquad\,\,\,\qquad=\frac{1}{\norm{v}_{H^{s}}}\Bigl|\int_{\R^{2}} (p''+ p' + p)(-\eta, \eta-\sigma, \sigma)\hat{v}(-\eta)\hat{u}(\eta-\sigma)\hat{u}(\sigma)\,d\sigma\,d\eta \\
&\qquad\qquad\,\,\,\quad\,- \int_{\R^{2}}\hat{v}(-\xi) n(\xi-\eta, \eta)\hat{u}(\xi-\eta)\hat{u}(\eta)\,d\eta\,d\xi\Bigr|.\\
\intertext{Renaming \( \eta \mapsto \xi, \sigma\mapsto \eta \) in the first integral, this is equal to}
&\frac{1}{\norm{v}_{H^{s}}}|\int_{\R^{2}} ((p''+ p' + p)(-\xi, \xi-\eta, \eta) - n(\xi-\eta, \eta))\hat{v}(-\xi)\hat{u}(\xi-\eta)\hat{u}(\eta)\,d\eta\,d\xi|.
\end{align*}
This expression is only zero as \(\norm{v}_{H^{s}}\to 0\) if
\begin{equation*}
    \begin{split}
n(\xi-\eta, \eta) &=(p+p'+p'')(-\xi, \xi-\eta, \eta)\\
 &=p(\xi-\eta, \eta, -\xi) + p(\eta, -\xi, \xi-\eta) + p(-\xi, \xi-\eta, \eta),
    \end{split}
\end{equation*}
which is what we wanted to show.

By \cref{prop_trilinear}, clearly
\begin{equation*}
    P(u,u,u) = P'(u,u,u) = P''(u,u,u).
\end{equation*}
Hence, 
\begin{align*}
    \mathcal{P}(u)&= \frac{1}{3}\intR P(u,u,u) + P'(u,u,u) + P''(u,u,u)\,dx\\
    &= \frac{1}{3}\int_{\R^2} (p+p'+p'')(-\xi, \xi-\eta, \eta)\hat{u}(-\xi)\hat{u}(\xi-\eta)\hat{u}(\eta)\,d\eta\,d\xi\\
    &= \frac{1}{3}\intR\hat{u}(-\xi)\intR n(\xi-\eta, \eta)\hat{u}(\xi-\eta)\hat{u}(\eta)\,d\eta\,d\xi\\
    &= \frac{1}{3}\intR uN(u,u)\,dx.
\end{align*}
\end{proof}

\section{Comparison with related equations}
\label{app:equations}
\subsection{A Whitham-Boussinesq system}
In \cite{dinvay2021}, the authors consider solitary-wave solutions of the Whitham-Boussinesq system
\begin{align*}
    L\eta_{t} + L v_{x} + (\eta v)_{x}= 0,\\      
    L v_{t} + \eta_{x} + v v_{x} = 0, 
\end{align*}
where \(L\) is a linear Fourier multiplier with symbol \(m(\xi)\) satisfying assumption very similar to \cref{assump:M}, but with \(s>\frac{1}{2}, s'>\frac{1}{2}\) and \(m(0)>0\) (as opposed to \(s>0, s'>\frac{1}{2}\) in Assumption \(\mathfrak{M}\)). In addition, the kernel of \(L^{-\frac{1}{2}}\) must satisfy certain integrability assumptions. 

They show existence of solitary-wave solutions to the above system by reducing it, through a set of changes of variables, to a single traveling wave equation,
\begin{equation}
    \lambda u + L^{-\frac{1}{2}}\left(\frac{(L^{-\frac{1}{2}}u)^{3}}{2}\right) + L^{-\frac{1}{2}}(L^{\frac{1}{2}}u {L^{-\frac{1}{2}}u}) + L^{\frac{1}{2}}\left(\frac{(L^{-\frac{1}{2}}u)^2}{2}\right) + Lu = 0,\label{eq:cubic}
\end{equation}
and apply Lion's concentration--compactness principle. This is equation \(1.27\) in \cite{dinvay2021}. 
If we ignore the cubic term, this equation is of the form 
\begin{equation}
    -\nu u + Mu - N(u,u) = 0,\label{myEq}
\end{equation}
with \(\nu = \lambda\), \(M=L\) and \(N\) a bilinear Fourier multiplier with symbol
\begin{equation*}
    n(\xi-\eta, \eta) =\frac{1}{2} \sum_{cyc} m(-\xi)^{\frac{1}{2}} m(\xi-\eta)^{-\frac{1}{2}} m(\eta)^{-\frac{1}{2}}.
\end{equation*}
We can verify that this satisfies our \cref{assump:N}, save a few technical details. Clearly, we must assume that \(n_{i} \in C^{1}( \mathbb{R})\), a slightly stronger assumption than the uniform continuity in \cite{dinvay2021}. Then, the bounds on \(n_{i}\) and the derivatives are satisfied for \(r_1 = \frac{s}{2}, r_{2} = r_{3} = -\frac{s}{2}\). 
Furthermore, 
\begin{align*}
    \sum_{i=1}^3 r_{i} = -\frac{s}{2} < s - \frac{1}{2}\,\, (>0),\\ 
    r_{i} + r_{j} \leq 0 < \frac{3s-1}{2} \,\, (>\frac{1}{4}).
\end{align*}
The only problem is that \(r_{1}\) is not strictly smaller than, but equal to \(\frac{s}{2}\). However, a close examination of our proof reveals that a strict inequality for \(\max_{i} r_{i}\) is only needed in the existence proof when \(\min_{i} r_{i}> \frac{s-1}{2}\) (at the end of the proof of Lemma 2.3).  When \(\min_{i} r_{i}\leq \frac{s-1}{2}\), which is the case here, \(\max_i r_{i} \leq \frac{s}{2}\) suffices. 

In the cubic term in \cref{eq:cubic}, all operators are of negative order, and the term could probably be treated with similar methods as in \cite{marstrander}.

\subsection{Babenko's equation}
Babenko's equation on finite depth modified to include surface tension, is given by
\begin{align}
   0 &= -u +\nu^{2}Ku - uKu - K(\frac{u^2}{2})\label{Babenko} \\
&\qquad+\beta \left(\frac{u'}{\sqrt{u'^2 + (1+Ku)^2}}\right)' - \beta K\left(\frac{1 + Ku}{\sqrt{u'^2 + (1 + Ku)^2}}-1 \right).\nonumber
\end{align}
This is eq. (17) in \cite{buffoni2004}. Here, \(\nu\) and \(\beta\) are parameters related to the wave speed and surface tension respectively. Setting \(\beta = 0\) gives Babenko's equation for pure gravity waves on {finite depth\cite{constantin2016}}. \(K\) is a linear Fourier multiplier given by
\[
\widehat{Ku}(\xi) = \xi \coth(\xi) \hat{u}(\xi) = k(\xi) \hat{u}(\xi)
.\] 
Compared with \cref{myEq}, the parameter \(\nu\) is in front of the wrong term, and the equation has non-polynomial nonlinearities. To compare it to equations like \cref{myEq}, we expand the \(\beta\)-terms in small \(u\) to second order, apply \(K^{-\frac{1}{2}}\) to both sides of the equation and make a change of variables \(u= K^{-\frac{1}{2}}v\). Then{
\begin{align*}
    0 &= K^{-1}v - \nu^{2}v + K^{-\frac{1}{2}}(K^{-\frac{1}{2}}v K^{\frac{1}{2}}v) + \frac{1}{2}K^{\frac{1}{2}}(K^{-\frac{1}{2}}v)^2\\
    &\quad - \beta K^{-1}v''+\beta K^{-\frac{1}{2}}\partial_{x}( K^{-\frac{1}{2}}v'K^{\frac{1}{2}}v) - \frac{\beta}{2} K^{\frac{1}{2}}(K^{-\frac{1}{2}}v')^2.
\end{align*}
This corresponds to 
\begin{align*}
    m(\xi) &= (1+\beta \xi^2)\frac{\tanh \xi}{\xi},\\
    n(\xi-\eta, \eta)&= \frac{1}{2}\sum_{cyc}k(-\xi)^{-\frac{1}{2}}k(\xi-\eta)^{-\frac{1}{2}}k(\eta)^{\frac{1}{2}}\\
    &\quad + \frac{1}{2}\beta \sum_{cyc} k(-\xi)^{-\frac{1}{2}}(-\xi) k(\xi-\eta)^{-\frac{1}{2}}(\xi-\eta)k(\eta)^{\frac{1}{2}},
\end{align*}}
where the sums are over cyclic permutations of \(-\xi, \xi-\eta, \eta\). 

For \(\beta>0\), we have \(s=1, s' = 2\). Then, the first sum in the expression for \(n\) is covered by our theory, but in {the last sum }\(n\) grows too quickly: \(r_{i} = \frac{1}{2}, i = 1,2,3\) and \(\sum_{i=1}^3 r_{i} = \frac{3}{2}\). Of course, the initial expansion of the \(\beta\)-terms only makes sense for small solutions, in which case some of our growth assumptions could probably be relaxed.

In the pure gravity case, \(\beta = 0\), the equation is exact, but it corresponds to negative \(s\), which is not covered by our theory.

\bibliographystyle{siam}
\bibliography{Paper1}


\end{document}